\documentclass[10pt,reqno]{amsart}
\usepackage[numbers]{natbib}
\usepackage{amsaddr}
\usepackage{etoolbox}
\patchcmd{\section}{\scshape}{\bfseries\scshape}{}{}
\makeatletter
\renewcommand{\@secnumfont}{\bfseries}
\makeatother

\makeatletter
\renewcommand\subsection{\@startsection{subsection}{2}%
  \z@{-.5\linespacing\@plus-.0\linespacing}{0.3\linespacing}%
  {\bfseries}}
\makeatother

\pdfoutput=1 
\usepackage[a4paper,margin=3cm]{geometry}

\usepackage[utf8]{inputenc} 
\usepackage[T1]{fontenc}    

\usepackage[final]{hyperref}      
\hypersetup{
    colorlinks=true,
    citecolor=green,
    filecolor=black,
    linkcolor=blue,
    urlcolor=blue
}

\usepackage{url}            
\usepackage{booktabs}       
\usepackage{amsfonts}       
\usepackage{nicefrac}       
\usepackage{microtype}      
\usepackage{amsthm}

\usepackage{enumitem}

\usepackage[dvips]{graphicx}
\usepackage{amssymb,amsmath,color}
\usepackage{cancel}
\usepackage{soul}

\usepackage{tikz-cd}
\usetikzlibrary{arrows.meta, 
                shapes.geometric}
\usetikzlibrary{matrix}

\usepackage[mathcal]{eucal}

\DeclareMathAlphabet\mathbfcal{OMS}{cmsy}{b}{n}

\newcommand{\vertiii}[1]{{\left\vert\kern-0.25ex\left\vert\kern-0.25ex\left\vert #1 
    \right\vert\kern-0.25ex\right\vert\kern-0.25ex\right\vert}}

\newcommand{\BEAS}{\begin{eqnarray*}}
\newcommand{\EEAS}{\end{eqnarray*}}
\newcommand{\BEA}{\begin{eqnarray}}
\newcommand{\EEA}{\end{eqnarray}}
\newcommand{\BEQ}{\begin{equation}}
\newcommand{\EEQ}{\end{equation}}
\newcommand{\BIT}{\begin{itemize}}
\newcommand{\EIT}{\end{itemize}}
\newcommand{\BNUM}{\begin{enumerate}}
\newcommand{\ENUM}{\end{enumerate}}
\newcommand{\BA}{\begin{array}}
\newcommand{\EA}{\end{array}}

\newtheorem{lemma}{Lemma}
\newtheorem{theorem}{Theorem}
\newtheorem{proposition}{Proposition}
\newtheorem{corollary}{Corollary}

\def \X{{\mathcal X}}

\def \A{{\mathcal A}}

\newcommand{\e}{\varepsilon}

\newcommand{\Ee}{\mathbb{E}}
\newcommand{\Rr}{\mathbb{R}}
\newcommand{\Pp}{\mathbb{P}}
\newcommand{\Nn}{\mathbb{N}}
\newcommand{\Zz}{\mathbb{Z}}

\newcommand{\cH}{\mathcal{H}}
\newcommand{\cA}{\mathcal{A}}

\newcommand{\cF}{\mathcal{F}}

\newcommand{\cX}{\mathcal{X}}

\newcommand{\cU}{\mathcal{U}}

\allowdisplaybreaks 

\makeatletter
\g@addto@macro{\endabstract}{\@setabstract}
\newcommand{\authorfootnotes}{\renewcommand\thefootnote{\@fnsymbol\c@footnote}}%
\makeatother

\title[A Non-asymptotic Analysis of Non-parametric  Temporal-Difference Learning]{A Non-asymptotic Analysis of Non-parametric Temporal-Difference Learning}

\begin{document}

\begin{center}
	\LARGE 
	A Non-asymptotic Analysis of Non-parametric \\ Temporal-Difference Learning \par \bigskip
	
	\normalsize
	Eloïse Berthier\textsuperscript{1}, Ziad Kobeissi\textsuperscript{1,2} and
	Francis Bach\textsuperscript{1} \par \bigskip
	
	\textsuperscript{1}Inria - Département d’informatique de l’ENS \\
	PSL Research University, Paris, France
	\smallskip \par
	\textsuperscript{2}Institut Louis Bachelier, Paris, France\smallskip \par 
\end{center}

\begin{abstract}
 Temporal-difference learning is a popular algorithm for policy evaluation.  In this paper, we study the convergence of the regularized non-parametric TD(0) algorithm, in both the independent and Markovian observation settings. In particular, when TD is performed in a universal reproducing kernel Hilbert space (RKHS), we prove  convergence of the averaged iterates to the optimal value function, even when it does not belong to the RKHS. We provide explicit convergence rates that depend on a source condition relating the regularity of the optimal value function to the RKHS. We illustrate this convergence numerically on a simple continuous-state Markov reward process. 
\end{abstract}

\section{Introduction}

One of the main ingredients of reinforcement learning (RL) is the ability to estimate the long-term effect on future rewards of employing a given policy. This building block, known as policy evaluation, already contains crucial features of more complex RL algorithms, such as SARSA or Q-learning~\cite{sutton2018reinforcement}. Temporal-difference learning (TD), proposed by~\cite{sutton1988learning}, is among the simplest algorithms for policy evaluation. The estimation of the performance of the policy is made through a value function. It is updated \textit{online}, after each new observation of a couple composed of a state transition and a reward.

Although the formulation of TD is quite  natural, its theoretical analysis has proved more challenging, as it combines two difficulties. The first one is that TD \textit{bootstraps}, in the sense that it uses its previous -- possibly inaccurate -- predictions to correct its next predictions, because it does not have access to a fixed ground truth. The second difficulty is that the observations are produced along a trajectory following a fixed policy (\textit{on-policy}), hence they are correlated, which calls for more involved stochastic approximation tools compared to independent identically distributed (\textit{i.i.d.}) samples. Moreover, using \textit{off-policy} samples, produced by a different policy than the one  being evaluated, can make the algorithm diverge~\cite{boyan1994generalization}. Off-policy sampling is out of our scope in this~paper.

A third element which is not inherent to TD further complicates the plot: function approximation. While TD was originally proposed in a tabular setting, its large-scale applicability has been greatly extended by its combination with parametric function approximation~\cite{bradtke1996linear}. This enables the use of any linear or non-linear function approximation method to model the value function, including neural networks. However, one can exhibit unstable diverging behaviors of TD even with simple non-linear approximation schemes~\cite{tsitsiklis1997analysis}. This combination of difficulties has been coined the ``deadly triad'' by~\cite{sutton2015introduction}. We argue that convergence can be obtained even with non-linear function approximation, by making use of the non-parametric formalism of reproducing kernel Hilbert spaces (RKHS), involving linear approximation in infinite-dimension. Studying this case could bring us closer to understanding what happens with other universal approximators used in practice, like neural networks.

\subsection{Contributions}

We study the policy evaluation algorithm TD(0) in the non-parametric case, first when the observations are sampled \textit{i.i.d.}~from the invariant distribution of the Markov chain resulting from the evaluated policy, and then when they are collected from a trajectory of the Markov chain with geometric mixing. In that sense we follow a similar outline as the analysis of~\cite{bhandari2018finite} which is dedicated to the linear case. 

The non-parametric formulation of TD closes the gap between the original tabular formulation and the parametric formulation which involves semi-gradients. It allows the use of classical tools and theory from kernel methods~\cite{cristianini2004kernel}. In particular, we highlight the central role of  infinite-dimensional covariance operators~\cite{baker1973joint, bach2022information} which already appear in the analysis of other non-parametric algorithms, like least-squares regression. We study a regularized variant of TD, a widely used way of dealing with misspecification in regression. Importantly, when the regularized TD approximation is run on an infinite-dimensional RKHS which is dense in the space of square-integrable functions, then there is no approximation error and the algorithm converges to the true value function. More precisely, we provide a proof of convergence in expectation of TD without approximation error, even when the true value function does not belong to the RKHS, under a weaker source condition. Furthermore, we give non-asymptotic convergence rates related to this source condition, which measures the regularity of the true value function relative to the RKHS, \textit{e.g.}, its smoothness if the RKHS is a Sobolev space~\cite{novak2018reproducing}.

 Note that using a universal kernel~\cite{micchelli2006universal} to obtain convergence of TD to the true value function is also interesting from a theoretical point of view. Indeed it exempts us from a possibly tedious study of the approximation (or projection) error on a given basis, and simply removes an error term which in general scales linearly with the horizon of the Markov reward process~\cite{mou2020optimal, yu2010error}.

In the rest of this section, we review the related literature. In Sec.~\ref{sec:td_algo}, we present the algorithm, along with generic results and notations. In Sec.~\ref{sec:mean_path}, we analyze a simplified version of the algorithm, namely population TD in continuous time. This allows to catch the main features of the analysis, while postponing the technicalities related to stochastic approximation. Sec.~\ref{sec:iid} is dedicated to the analysis of non-parametric TD with \textit{i.i.d.}~observations, while Sec.~\ref{sec:markov} consists in a similar analysis for correlated observations sampled from a geometrically mixing Markov chain. Finally, in Sec.~\ref{sec:numerics}, we present simple numerical simulations illustrating the convergence results and the role the main parameters.

\subsection{Related literature}

\textbf{Temporal-difference learning.} The TD algorithm was introduced in its tabular version by~\cite{sutton1988learning}, with a first convergence result for linearly independent features, later extended to dependent features by~\cite{dayan1992convergence}. Further stochastic approximation results were proposed by~\cite{jaakkola1993convergence} for the tabular case, and by~\cite{schapire1996worst} for the linear approximation case.  \cite{tsitsiklis1997analysis} provided a thorough asymptotic analysis of TD with linear function approximation, while failure cases were already known~\cite{baird1995residual}. A non-asymptotic analysis was later proposed by~\cite{lakshminarayanan2018linear} in the \textit{i.i.d.}~sampling case with constant step size, concurrently to another approach extending to Markov sampling by~\cite{bhandari2018finite}. Other  problem-dependent bounds for linear TD were derived around the same period~\cite{dalal2017finite, srikant2019finite}, along with an analysis of variance-reduced TD~\cite{korda2015td, xu2020reanalysis}. All of the analyses mentioned above focus either on the tabular or on the linear \textit{parametric} TD algorithm. A recent work by~\cite{long20212} deals with the batch counterpart of non-parametric TD, namely the least-squares TD algorithm (LSTD), but they rather focus on the analysis of the statistical estimation error. Importantly, LSTD only requires offline computations and is not related to stochastic approximation. Most closely related to our work is the non-parametric regularized TD setting studied by~\cite{koppel2020policy}. However, their analysis is limited to the case where the optimal value function belongs to the RKHS. This is not sufficient to get rid of the approximation error term. Also, we will show later that regularization is not necessary in this case. Furthermore, their analysis is restricted to the \textit{i.i.d.}~setting, for which we will require fewer regularity assumptions.

\textbf{Kernel methods in RL.} To tackle large-dimensional problems, kernel methods have been combined with various RL algorithms, including approximate dynamic programming~\cite{ormoneit2002kernel, bhat2012non, barreto2011reinforcement, grunewalder2012modelling}, policy evaluation~\cite{dai2017learning},  policy iteration~\cite{farahmand2016regularized}, LSTD~\cite{long20212}, the linear programming formulation of RL~\cite{dietterich2001batch}, upper confidence bound~\cite{domingues2021kernel}, or fitted Q-iteration~\cite{duan2021optimal}. Such kernel methods often come along with practical ways to reduce the computational complexity that grows with the number of observed  transitions and rewards~\cite{barreto2016practical, koppel2020policy}.

\textbf{Stochastic approximation.} The analysis of TD requires tools from stochastic approximation~\cite{benveniste2012adaptive}, among which the ODE method~\cite{borkar2000ode}. Such tools are primarily designed for finite-dimensional problems. Stochastic gradient descent (SGD)~\cite{bottou2018optimization} is a specific instance of stochastic approximation that has received extensive attention for supervised learning. In particular, the role of regularization of SGD for least-squares regression has been  studied~\cite{caponnetto2007optimal, cucker2007learning}, as well as the effect of of sampling data from a Markov chain~\cite{nagaraj2020least}. Finally, we use a formalism which is close to the analyses  \cite{dieuleveut2016nonparametric, pillaud2018exponential, berthier2020tight} of non-parametric SGD for least squares regression.

\section{Problem formulation and generic results} \label{sec:td_algo}

\subsection{The non-parametric TD(0) algorithm}

We consider a Markov reward process (MRP), \textit{i.e.}, a Markov chain with a reward associated to each state. This is what results from keeping the policy fixed in a Markov decision process (MDP) for policy evaluation. We consider MRPs in discrete-time, but not necessarily with a countable state space~$\cX$. Specifically, we use the formalism of Markov chains on a measurable state space which unifies discrete- and continuous-state Markov chains.   Formally, let $\cX \subset \Rr^d$ a measurable set associated with the $\sigma$-algebra $\cA$  of Lebesgue measurable sets. Let $(x_n)_{n \geq 1}$ a time-homogeneous Markov chain with Markov kernel~$\kappa$. A Markov kernel~\cite{reiss2012course, klenke2013probability} is a mapping $\kappa:\cX \times \cA \rightarrow [0, 1]$ that has the following two properties: (1) for every $x \in \cX$, $\kappa(x, \cdot)$ is a probability measure on~$\cA$, and (2) for every $A \in \cA$, $\kappa(\cdot, A)$ is $\cA$-measurable. If $\cX$ is a countable set, $\kappa$ is represented by a transition matrix~$Q$ such that $Q_{i,j} := \Pp(j | i) = \kappa(i, \{j\})$, for any $i, j \in \cX$.

We define a reward function $r:\cX \rightarrow \Rr$ uniformly bounded by $R<\infty$, and a discount factor~${\gamma \in [0, 1)}$. The aim of policy evaluation is to compute the value function of the MRP:
\begin{align}
    \forall x \in \cX, \quad V^*(x) = \Ee \Big[\sum_{n =0}^{+\infty} \gamma^n r(x_n) ~\Big|~ x_0 = x \Big], \label{eqn:value}
\end{align}
where the $(x_n)_{n \geq 1}$ are drawn from the Markov chain. A probability distribution $p : \A \rightarrow  \Rr$ is a stationary distribution for~$\kappa$ if for all $A \in \cA$, $p(A) = \int_\cX \kappa(x, A) p(dx)$. 
The existence and uniqueness of a stationary distribution~$p$, along with the convergence of the Markov chain to~$p$ in total variation, is ensured by ergodicity conditions. A sufficient condition is that the Markov chain is Harris ergodic, \textit{i.e.}, it has a regeneration set, and is aperiodic and positively recurrent (see~\cite{asmussen2003applied} and~\cite{durrett2019probability} for an exposition of Harris chains). For discrete-state Markov chains, ergodicity conditions can be expressed somewhat more simply, and any aperiodic and positive recurrent Markov chain has a unique invariant distribution. Throughout this paper, we assume that~$p$ is the unique invariant distribution of the Markov chain, and that it has full support on~$\cX$. Only in Sec.~\ref{sec:markov}, we will in addition assume that the Markov chain is geometrically mixing. 

We define~$L^2(p)$, the set of squared integrable functions~$f:\cX \rightarrow \Rr$ with respect to~$p$, with the norm $\|f\|_{L^2(p)}^2 = \int_\cX f(x)^2 p(dx) < + \infty$. 
We also consider a reproducing kernel Hilbert space~$\cH$ of $\cA$-measurable functions, associated to a positive-definite kernel~$K : \cX \times \cX \rightarrow \Rr$. For all $x \in \cX$, we use the notation~$\Phi(x) := K(x, \cdot)$ for the mapping of~$x$ in~$\cH$, and $\langle \cdot, \cdot \rangle_\cH$ for the inner product in~$\cH$ (we sometimes drop the index). We assume that $M_\cH := \sup_{x \in \cX} K(x, x)$ is finite, which implies that $\cH \subset L^2(p)$. More precisely,  the $\cH$-norm controls the $L^2(p)$-norm: any sequence converging in~$\cH$ thus converges in~$L^2(p)$. Indeed, if $f \in \cH$:
\begin{equation}
 \textstyle \|f \|_{L^2(p)}^2  = \int f(x)^2 dp(x) = \int \langle f,  \Phi(x) \rangle_\cH ^2 dp(x) \leq  \| f\|^2_\cH \int \|\Phi(x)\|^2_\cH dp(x) \leq M_\cH \|f \|^2_\cH. \label{eqn:norms}
\end{equation}
We also assume that $r \in L^2(p)$. The non-parametric TD(0) algorithm in the RKHS~$\cH$ is defined as follows~\cite{ormoneit2002kernel, koppel2020policy}. Draw a sequence $(x_n)_{n\geq 0}$ according to the Markov chain with initial distribution~$p$, and collect the corresponding rewards $(r(x_n))_{n \geq 0}$. Define a sequence of non-negative step sizes $(\rho_n)_{n \geq 1}$. We build recursively a sequence of approximate value functions $(V_n)_{n \geq 0}$ in $L^2(p)$. Throughout the paper, we take $V_0 = 0$ for simplicity, but note that all the results can be adapted  to the case $V_0 \in \cH$. For $n \geq 1$:
\begin{align}\forall y \in \cX, \quad V_{n}(y) = V_{n-1}(y) + \rho_{n} \Big[ r(x_{n}) + \gamma V_{n-1}(x_{n}') - V_{n-1}(x_{n}) \Big] K(x_{n}, y) , \label{eqn:fct_value} \end{align}
where $x'_n := x_{n+1}$. The term in brackets is called a temporal-difference. Equivalently, in the RKHS:
\begin{align}
    V_{n} = V_{n-1} + \rho_{n} \Big[ r(x_{n}) + \gamma V_{n-1}(x_{n}') - V_{n-1}(x_{n}) \Big] \Phi(x_{n}). \label{eqn:TD}
\end{align}
This update has a running time complexity of $O(n^2)$, which can be improved to $O(n)$, \textit{e.g.}~using Nyström approximation or random features~\cite{halko2011finding}. More details on the implementation are given in App.~\ref{subsec:implem}. This non-parametric formulation is a natural extension of the tabular TD algorithm. Indeed, if $\cX$ is a countable set and $K(x, y) = \mathbf{1}_{x = y}$ is a Dirac kernel -- a valid positive-definite kernel -- then we  exactly recover tabular TD: the update rule~(\ref{eqn:fct_value}) becomes, after observing a transition $(i, i', r_i):=(x_n, x'_n, r(x_n))$:
\begin{align}
     \quad V_{n}(i) = V_{n-1}(i) + \rho_{n} \Big[ r_i + \gamma V_{n-1}(i') - V_{n-1}(i) \Big] ,  \quad \text{and~ } \forall j \neq i, ~ V_{n}(j) = V_{n-1}(j). \label{eqn:tab_TD}
\end{align}
This also covers the \textit{semi-gradient} formulation of TD for linear function approximation~\cite{sutton2018reinforcement}. Suppose~$\cH$ has finite dimension~$d$, then $V_n$ can be identified to $\xi_n \in \Rr^d$, and we are searching for an approximation of the form $V_n(x) = \xi_n^\top \Phi(x)$. Then \eqref{eqn:TD} becomes:
\begin{align}
    \xi_n = \xi_{n-1} + \rho_n  \Big[ r(x_{n}) + \gamma V_{n-1}(x_{n}') - V_{n-1}(x_{n}) \Big] \nabla_\xi V_n(x_{n}). \label{eqn:lin_TD}
\end{align}
Since $V_0\in \cH$, all the iterates $V_n$ are in the RKHS, in particular $V_n \in \text{span}\{\Phi(x_k) \}_{1\leq k \leq n}$. Consequently, if the sequence $(V_n)$ converges in the topology induced by the $L^2(p)$-norm, it converges in $\overline \cH$, the closure of $\cH$ with respect to the $L^2(p)$-norm. 
In particular, for a dense RKHS and because~$p$ has full support on~$\cX$, $\overline \cH= L^2(p)$, but in general it only holds that~$\overline \cH \subset L^2(p)$.

To understand the behavior of the algorithm, we will first consider the \textit{population} version  (also called \textit{mean-path} in~\cite{bhandari2018finite}) of the algorithm: set $V_0= 0$ and for $n \geq 1$:
\begin{align}
    V_n = V_{n-1} + \rho_{n} \Ee_{(x, x')\sim q} \left[ \left(r(x) + \gamma  V_{n-1}(x') -  V_{n-1}(x) \right) \Phi(x) \right] , \label{eqn:mean_TD}
\end{align}
where the expectation is taken with respect to $q(dx, dx'):= p(dx) \kappa(x, dx')$. Since $V_{n-1} \in \cH$, the reproducing property holds: $V_{n-1}(x) = \langle V_{n-1}, \Phi(x) \rangle_\cH$. Hence the update is affine and reads: $V_n = V_{n-1} + \rho_n (A V_{n-1} +b)$, with $A := \Ee_q \left[\gamma \Phi(x) \otimes \Phi(x') - \Phi(x) \otimes \Phi(x) \right]$ and $b := \Ee_p \left[ r(x) \Phi(x)\right]$, where $\otimes$ denotes the outer product in~$\cH$ defined by $g \otimes h : f \mapsto \langle f, h \rangle_\cH g$.

\subsection{Covariance operators}

Assume that the expectations $\Sigma:=\Ee_p[\Phi(x) \otimes \Phi(x)]$
and $\Sigma_1 := \Ee_{q}[\Phi(x) \otimes \Phi(x')]$ are well-defined. $\Sigma$ and $\Sigma_1$ are the uncentered auto-covariance operators of order 0 and 1 of the Markov process $(x_n)_{n \geq 1}$, under the invariant  distribution~$p$. They are operators from $\cH$ to $\cH$, such that, for all $f, g \in \cH$, using the reproducing property:
\begin{align}
\begin{split}
    \Ee_p [f(x)g(x)] &= \Ee_p[ \langle f, \Phi(x) \rangle_\cH \langle g, \Phi(x) \rangle_\cH ] = \langle f, \Ee_p[ \langle  g, \Phi(x) \rangle_\cH \Phi(x)]  \rangle_\cH = \langle f, \Sigma g \rangle_\cH   \\
    \Ee_q [f(x)g(x')] &=\Ee_q[ \langle f, \Phi(x) \rangle_\cH \langle g, \Phi(x') \rangle_\cH ] = \langle f, \Ee_p[ \langle  g, \Phi(x') \rangle_\cH \Phi(x)]  \rangle_\cH = \langle f, \Sigma_1 g \rangle_\cH. \label{eqn:cov_def}
\end{split}
\end{align}
In particular, for all $y \in \cX$ and $f \in \cH$, $(\Sigma f)(y) = \langle \Phi(y), \Sigma f \rangle_\cH = \Ee_p[f(x) K(x, y)]$ and similarly, $(\Sigma_1 f)(y) = \Ee_q[f(x') K(x, y)]$. Following~\cite{dieuleveut2016nonparametric}, $\Sigma$ and $\Sigma_1$ can therefore be extended to operators $\Sigma^e$ and $\Sigma_1^e$ from $L^2(p)$ to $L^2(p)$ defined by:
\begin{align}
\begin{split}
    \Sigma^{e} : f \mapsto \int_\cX f(x) \Phi(x) p(dx), ~ &\text{such that } \forall y\in \cX,~ (\Sigma^{e} f)(y) = \Ee_p[f(x)K(x, y)]  \\
    \Sigma_1^{e} : f \mapsto \iint_{\cX^2} f(x') \Phi(x) q(dx, dx'), ~ &\text{such that } \forall y\in \cX,~ (\Sigma_1^{e} f)(y) = \Ee_q[f(x')K(x, y)]. \label{eqn:cov_ext}
\end{split}
\end{align}
These two operators are the  building blocks of the TD iteration~(\ref{eqn:mean_TD}). In particular, $A = \gamma \Sigma_1 - \Sigma$ and $b= \Sigma^e r$, the latter being valid for $r \in L^2(p)$. With a slight abuse of notation, we denote simply as $\Sigma$, $\Sigma_1$ the extended operators. Furthermore~\cite{dieuleveut2016nonparametric}, $\text{Im}(\Sigma) \subset \cH$ and $\Sigma^{1/2}$ is an isometry from $L^2(p)$ to $\cH$: 
\begin{align}
    \forall f \in \overline{\cH}, \quad \|f \|_{L^2(p)} = \| \Sigma^{1/2} f \|_\cH. \label{eqn:isometry}
\end{align}

The fact that~$p$ is a stationary distribution for~$\kappa$ implies a particular constraint linking~$\Sigma$ and $\Sigma_1$:
\begin{lemma}
There exists a unique bounded linear operator $\tilde \Sigma_1 : \cH \rightarrow \cH$ such that  $\Sigma_1 = \Sigma^{1/2} \tilde \Sigma_1 \Sigma^{1/2}$ on $\overline{\cH}$, and $\| \tilde \Sigma_1 \|_{\rm op} \leq 1$ ($\| \cdot \|_{\rm op}$ is the $\cH$-operator norm). \label{prop1}
\end{lemma}
 This results from an application of~\cite[Thm.~1]{baker1973joint}, valid on~$\cH$ and extended by continuity to~$\overline{\cH}$. See also~\cite{fukumizu2004dimensionality} for an exposition of cross-covariance operators specifically in an RKHS. In finite dimension, this is retrieved with generic results on positive semi-definite (PSD) matrices. Specifically, if $\cH \subset \Rr^m$, the uncentered covariance matrix of the random variable $(\Phi(x), \Phi(x'))$, when $(x, x') \sim q$~is:
\begin{center}$\begin{pmatrix} \Sigma & \Sigma_1 \\
 \Sigma_1^\top & \Sigma \end{pmatrix} \succeq 0. $\end{center}
 
Using a classical condition on block matrices~\cite[Prop.~1.3.2]{bhatia2013matrix}, this matrix is PSD if and only if there exists a matrix $\tilde \Sigma_1$ such that $\|\tilde \Sigma_1\|_{\rm op} \leq 1$ and $\Sigma_1 = \Sigma^{1/2} \tilde \Sigma_1 \Sigma^{1/2}$ ($\| \cdot \|_{\rm op}$ is also the spectral norm in this case). This corresponds to the fact that the Schur complement of a PSD block matrix is also PSD. 

\textbf{Assumptions on~$\Sigma$ and~$V^*$.} We assume that $x\mapsto K(x, x)$ is uniformly bounded by~$M_\cH$. Therefore, the eigenvalues of~$\Sigma$ are upper-bounded. However, unlike~\cite{tsitsiklis1997analysis} and \cite{bhandari2018finite}, we do not assume them to be lower-bounded, \textit{i.e.},~$\Sigma \succeq 0$ is not invertible in general. We will formulate our convergence results for two sets of  assumptions.   The first one recovers known results from~\cite{bhandari2018finite} for linear function approximation. The second one assumes that~$V^*$ verifies a \textit{source condition}~\cite[Chap.~1]{dieuleveut2017stochastic}:
\setlist{nolistsep} \begin{enumerate}[label=\textbf{(A\arabic*)}]
    \item \label{hypo:?} $V^* \in \cH$, $\cH$ is finite-dimensional and $\Sigma$ has full-rank;
    \item \label{hypo:??}$V^* \in \Sigma^{\theta/2}(\cH)$ for some $\theta \in (-1, 1]$ (and consequently, $\| \Sigma^{-\theta/2} V^*\|_\cH<+\infty$), and $\overline{\cH} = L^2(p)$ (\textit{i.e.}, $K$ is a universal kernel).
\end{enumerate}
In \ref{hypo:?}, $\cH$ is finite-dimensional because $\Sigma$ cannot be simultaneously compact ($x\mapsto K(x, x)$ being uniformly bounded) and invertible in infinite-dimension~\cite{cheney2001analysis}.  Recalling the isometry property~(\ref{eqn:isometry}), the case $\theta=-1$ always holds in \ref{hypo:??} because $V^* \in L^2(p)$ (which we prove in the next subsection). The case~$\theta = 0$ is equivalent to $V^* \in \cH$. For $\theta >0$, it must be interpreted as: $\| \Sigma^{-\theta/2} V^* \|_\cH^2 := \inf \{ \| V \|^2_\cH ~|~ V \text{ s.t. } V^* = \Sigma^{\theta/2} V \}$, with $\| \Sigma^{-\theta/2} V^* \|_\cH = + \infty$ if $V^* \notin \Sigma^{\theta/2} (\cH)$. Using a universal approximation removes the need for a projection operator on $\overline{\cH}$, as typically used for finite-dimensional function approximation, and hence there will be no projection error~\cite{tsitsiklis1997analysis}.

\subsection{Non-expansiveness of the Bellman operator} \label{subseq}

It is known that the value function~$V^*$ of the MRP is a fixed point of the Bellman operator $T$. We define two operators $P$ and $T: L^2(p) \rightarrow L^2(p)$ by, for $V \in L^2(p)$, 
    $PV(x) = \Ee_{x'\sim \kappa(x, \cdot)} V(x')$  and $TV(x) = r(x) + \gamma PV(x)$.
Both operators can be expressed in terms of~$\Sigma$ and $\Sigma_1$. For $V \in L^2(p)$:
\begin{align}\left\{
    \begin{array}{l}
    \Sigma PV = \Ee_p[\Phi(x) (PV)(x)] = \Ee_q [\Phi(x) V(x')] = \Sigma_1 V \\
    \Sigma TV = \Sigma r + \gamma \Sigma_1 V.
    \label{eqn:bellman_sigma}
    \end{array}\right.
\end{align}
\begin{lemma}\label{prop2}
For any $V \in L^2(p)$: $    \| PV \|_{L^2(p)} \leq \| V \|_{L^2(p)} $.
\end{lemma}
This is a direct reformulation of~\cite[Lemma 1]{tsitsiklis1997analysis}, the proof of which is given in  App.~\ref{sec:proofs_gen}. 
As stressed by~\cite{tsitsiklis1997analysis}, this strongly relies on the fact that~$p$ is a stationary distribution of the Markov chain. It implies that~$T$ is a $\gamma$-contraction mapping on~$L^2(p)$ and has as unique fixed point~$V^*$. One can check that if $\Sigma$ is non-singular, Lemma~\ref{prop2} is exactly equivalent to $\| \Sigma^{-1/2} \Sigma_1 \Sigma^{-1/2} \|_{\rm op} \leq 1$, that is, Lemma~\ref{prop1}. Moreover, using Lemma~\ref{prop2}, we obtain $
    \| V^* \|_{L^2(p)} \leq \| r\|_{L^2(p)}/(1-\gamma)$ and $V^* \in L^2(p)$.

\section{Analysis of a continuous-time version of the population TD algorithm} \label{sec:mean_path}

Before considering regularized TD with stochastic samples, we look at simplified versions of the algorithm that  momentarily remove the difficulties related to stochastic approximation. Specifically, we consider the population version of TD to capture a ``mean'' behavior, and a continuous-time algorithm to avoid choosing step sizes. Instead, we  focus on the role of the regularization parameter.

\vspace{-0.5em}

\subsection{Existence of a fixed-point for regularized TD}

\vspace{-0.5em}

For $\lambda\geq 0$, let us consider the regularized population recursion:
\begin{align} V_n = V_{n-1} + \rho_{n} (\Sigma r +(\gamma \Sigma_1 - \Sigma - \lambda I)  V_{n-1} ). \label{eqn:lbda_TD}
\end{align}
If the TD iterations converge, its limit will be a solution of the \textit{regularized} fixed point equation: \begin{align}
    \Sigma r + (\gamma \Sigma_1 - \Sigma - \lambda I) V = 0. \label{eqn:fixed}
\end{align}

\begin{proposition} 
If $\lambda >0$, then $\gamma \Sigma_1 - \Sigma - \lambda I $ is non-singular on~$\cH$ and the fixed point equation~(\ref{eqn:fixed}) admits a unique solution $V^*_\lambda$ in $L^2(p)$, defined by $ V^*_\lambda = (\gamma \Sigma_1 - \Sigma - \lambda I )^{-1} \Sigma r $. Furthermore, $V^*_\lambda \in \cH$~and:\label{prop3} 
\begin{align}
    \|V^*_\lambda \|_\cH \leq \frac{\| \Sigma r \|_\cH}{\lambda} \leq \frac{\sqrt{M_\cH}\|  r \|_{L^2(p)}}{\lambda }. \label{eqn:lbda_norm}
\end{align} 
\end{proposition}
The proof is in App.~\ref{sec:proofs_pop}. 
Hence, for $\lambda>0$, the $\cH$-norm of $V^*_\lambda$ is always bounded, unlike $\|V^*\|_\cH$.

\subsection{Convergence of the regularized fixed point to the optimal value function}
Recalling that~$V^* \in L^2(p)$,
it satisfies the relation $TV^* = V^*$,
implying that $\Sigma TV^* = \Sigma V^*$, \textit{i.e.}, $\Sigma r + (\gamma \Sigma_1 - \Sigma) V^* = 0$. This \textit{unregularized} fixed point equation  possibly has other solutions, but if~$K$ is a universal kernel, as assumed by~\ref{hypo:??}, then~$\Sigma$ is injective~\cite{steinwart2001influence} and~$V^*$ is the unique solution.  Let us recall that~\ref{hypo:??} does not
imply that $V^*$ has a bounded $\cH$-norm. 
However,  we can control the $L^2(p)$-norm of $V^*_\lambda - V^*$ when $\lambda$ is small using the \textit{source condition} \ref{hypo:??}.

\begin{proposition} \label{prop4}
Assume that $\lambda >0$ and assumption~\ref{hypo:??}.
Then:
\begin{align}
    \| V_\lambda^* - V^* \|_{L^2(p)}^2 \leq \frac{\lambda^{\theta+1} }{(1-\gamma)^2} \| \Sigma^{-\theta/2} V^*\|_\cH^2. \label{eqn:cv_lambda}
\end{align}
\end{proposition}
The proof in App.~\ref{sec:proofs_pop} 
is inspired by similar results~\cite{caponnetto2007optimal, cucker2007learning} in the context of ridge regression (recovered for $\gamma=0$). 
Note that only $\|V^*_\lambda - V^*\|_{L^2(p)}$ is controlled, not $\|V^*_\lambda - V^*\|_\cH$. 
Consequently, we obtain the convergence of $V^*_{\lambda}$ to $V^*$ in $L^2(p)$-norm when $\lambda \rightarrow 0$: the higher $\theta$ is, the faster the rate of convergence. For universal Mercer kernels~\cite{cucker2002mathematical}, if we drop the source condition \ref{hypo:??}, using only the fact that $V^* \in L^2(p)$ -- corresponding to $\theta=-1$ in \ref{hypo:??} -- we can still prove that $V^*_\lambda$ converges to $V^*$ in $L^2(p)$-norm when $\lambda \rightarrow 0$, but without an explicit rate (see App.~\ref{sec:proofs_pop}, Cor.~\ref{cor_mercer}).

\subsection{Convergence of continuous-time population TD}

Following the ordinary differential equation (ODE) method~\cite{borkar2000ode}, we study the continuous-time counterpart of the population iteration~(\ref{eqn:lbda_TD}).
 At least formally, this consists in
defining $\widetilde{V}_t$=$V_{n(t)}$
for $t$ and $n(t)$ satisfying  $t= \sum_{i=1}^{n(t)}\rho_i$,
and letting $\rho_i$ tend to $0$ for any $i\geq1 $,
where $V_{n(t)}$ is defined by recursion using \eqref{eqn:lbda_TD}.
With a slight abuse of notation, we use the notation $V_t$ instead of $\widetilde{V}_t$.
We then obtain the following ODE in~$\cH$: $V_0 = 0$ and for $t \geq 0$:
\begin{align}
    \frac{d  V_t}{dt} = (A - \lambda I)  V_t + b. \label{eqn:ode}
\end{align}
We can exhibit a Lyapunov function for this dynamical system, see~\cite{slotine1991applied}.
This implies that $V_t$ converges to $V^*_\lambda$ when $t$ tends to infinity,
where $V^*_\lambda$ is defined in Prop.~\ref{prop3}. 
More precisely, for $\beta \in \{-1, 0\}$,
we define $W^{\beta}$, the Lyapunov function, by $W^\beta(t) := \|\Sigma^{-\beta/2} (V_t - V^*_\lambda) \|_\cH^2$ (please note that~$\beta$'s role in $W^\beta$ is an index, not a power). $W^0(t)$ strictly decreases with $t$ as follows:

\begin{lemma}[Descent Lemma] \label{prop5} For $\lambda > 0$, for all $t \geq 0$, the following holds:
\begin{align}
    \frac{dW^0(t)}{dt} \leq - 2(1- \gamma) W^{ - 1}(t) - 2\lambda W^0(t), \label{eqn:lyap}
\end{align}
\end{lemma}
The proof mainly relies on the contraction property of the Bellman operator (see App.~\ref{sec:proofs_pop}). We can then deduce the convergence of the ODE~(\ref{eqn:ode}) to $V^*_\lambda$.

\begin{proposition} \label{prop6}
Under assumption~\textup{\ref{hypo:?}}, the solution~$V_t$ of the ODE~(\ref{eqn:ode}) with~$\lambda = 0$ is such that:
\begin{align}
  \text{For ~}T >0, \quad \|\overline V_T - V^* \|^2_{L^2(p)} \leq \frac{1}{2(1-\gamma)} \frac{\|V^*\|_\cH^2}{T}, \label{eqn:easy_cv}
\end{align}
where $\overline V_T$ is the Polyak-Ruppert average~\cite{polyak1992acceleration} of~$V_t$, defined by $\overline V_T :=\frac{1}{T} \int_0^T V_t dt  $.

Under assumption~\textup{\ref{hypo:??}}, the solution~$V_t$ of the ODE~(\ref{eqn:ode}) with~$\lambda > 0$ is such that:
\begin{align}
  \text{For ~}T \geq 0, \quad   \|V_T - V_\lambda^* \|^2_\cH \leq \|V^*_\lambda\|^2_\cH e^{-2  \lambda T}. \label{eqn:fast_cv}
\end{align}
\end{proposition}
 Under~\ref{hypo:?}, we recover the same~$O(1/T)$ convergence rate as~\cite{bhandari2018finite}.  We focus on~\ref{hypo:??}, where we get a fast convergence to $V^*_\lambda$ in~$\cH$-norm (stronger than  $L^2(p)$). However, we are rather interested in convergence to~$V^*$. Prop.~\ref{prop4} quantifies how far $V^*_\lambda$ is from~$V^*$. Indeed, the error decomposes as:
\begin{align}
    \| V_T - V^* \|_{L^2(p)}^2 \leq 2 M_\cH \| V_T - V^*_\lambda \|_\cH^2 + 2 \| V_\lambda^* - V^* \|_{L^2(p)}^2.
    \label{eqn:error_dcp}
\end{align}
Combining Propositions~\ref{prop3}, \ref{prop4}, \ref{prop6} shows a trade-off on~$\lambda$: $\|V_T - V^* \|^2_{L^2(p)} = O\left( e^{-2\lambda T}/\lambda^2 + \lambda^{\theta+1} \right)$. Taking  $\lambda=(3+\theta) \log T/(2T)$ balances the terms up to logarithmic factors: $\|V_T - V^* \|^2_{L^2(p)} = \tilde O\left( T^{-1-\theta} \right)$ (where $\tilde O(g(n)) := O(g(n) \log(n)^\ell)$, for some $\ell \in \Rr$). In particular, for $\theta=0$, \textit{i.e.},~$V^* \in \cH$, we recover a convergence rate $\tilde O\left( 1/T \right)$: up to logarithmic factors, it is the same as the  unregularized case with averaging, assuming \ref{hypo:?}. In this case, regularization brings no benefits.

\section{Stochastic TD with \textit{i.i.d.}~sampling} \label{sec:iid}

We now consider stochastic TD iterations~(\ref{eqn:TD}), where the couples $(x_n, x'_n)_{n \geq 1}$ are sampled \textit{i.i.d.}~from the distribution $q(dx, dx') = p(dx) \kappa(x, dx')$.  Such \textit{i.i.d.}~samples can be obtained by running the Markov chain until it has mixed so that $x_n \sim p$, collecting a couple $(x_n, x'_n)$, and restarting. With $A_n := \gamma \Phi(x_n) \otimes \Phi(x'_n) - \Phi(x_n) \otimes \Phi(x_n)$ and $b_n := r(x_n) \Phi(x_n)$, we study the recursion:
\begin{align}
    V_{n} = V_{n-1} + \rho_{n}( (A_{n} - \lambda I) V_{n-1} + b_{n}). \label{eqn:iid_TD}
\end{align}
In particular, $\Ee_q[A_n] = A$, $ \Ee_p[b_n] = b$, and $A_n$ and $b_n$ are independent of the past $(V_k)_{k < n}$. For~$\beta \in \{0, 1\}$, let $W_n^\beta := \| \Sigma^{-\beta/2}(V_n - V^*_\lambda)\|_\cH^2$. Adapting the proof of Lemma~\ref{prop5}, we exhibit a similar decreasing behavior of $W_n^0$ in expectation, hence showing that $\Ee [\|V_n - V^*_\lambda \|_\cH^2] \rightarrow 0$ for well-chosen step sizes~$\rho_n$. Finally, $\lambda$ is chosen to balance $\Ee [\|V_n - V^*_\lambda \|_{L^2(p)}^2]$ and $\|V^*_\lambda - V^* \|_{L^2(p)}^2$. 
We define
$V^{\textit{(e)}}_n$ and $V^{\textit{(t)}}_n$ as the  exponentially-weighted and the tail-averaged $n$-th iterates respectively:
\begin{align}
    V_n^{\textit{(e)}} := \frac{\sum_{k=1}^n  (1-\rho \lambda)^{n-k}  V_{k-1}}{\sum_{k=1}^n (1-\rho \lambda)^{n-k}} \text{~~~~and~~~~} V_n^{\textit{(t)}} := \frac{1}{n - \lfloor n/2 \rfloor +1}\sum_{k=\lfloor n/2 \rfloor}^{n}  V_{k-1}. \label{eqn:averages}
\end{align}
\begin{theorem} \label{thm1}Let $n \geq 9$.
Under assumption~\textup{\ref{hypo:??}} with $-1 < \theta \leq 1$, there exist a positive real number~$\underline{\lambda}_\theta$ independent of~$n$ such that, for $\lambda_0\geq\underline{\lambda}_{\theta}$,
\setlist{nolistsep} \begin{enumerate}[label=(\alph*), noitemsep]
\item
\label{hypo:thm1_a}
Using $\lambda = \lambda_0 n^{-\frac{1}{3+\theta}}$ and a  constant step size $\rho = \frac{\log n}{\lambda n}$, then: 
$$ \Ee [ \|V_n - V^* \|_{L^2(p)}^2 ] = O((\log n) n^{-\frac{1+\theta}{3+\theta}}). $$
\item
\label{hypo:thm1_b}
Using $\lambda = \lambda_0 n^{-\frac{1}{2+\theta}}$ and a constant step size $\rho = \frac{\log n}{\lambda n}$, then:
$$ \Ee [ \| V_n^{\textit{(e)}} - V^* \|_{L^2(p)}^2 ] = O((\log n) n^{-\frac{1+\theta}{2+\theta}}). $$
\item
\label{hypo:thm1_c}
Using $\lambda = \lambda_0 n^{-\frac{1}{2+\theta}}$ and a constant step size $\rho = \frac{2 \log n}{\lambda n}$ for the first $\lfloor n/2 \rfloor -1$ iterates and then a decreasing step size $\rho_k = \frac{1}{\lambda k}$, then:
$$ \Ee [ \| V_n^{\textit{(t)}} - V^* \|_{L^2(p)}^2 ] = O((\log n) n^{-\frac{1+\theta}{2+\theta}}). $$
\end{enumerate}
\end{theorem}

A similar exponentially-weighted averaging scheme as in \ref{hypo:thm1_b} has been used to study constant step size SGD in~\cite{defossez2017adabatch}. When~$\gamma=0$, the rates can be compared to existing results for SGD. For example, for $\theta \in [0,1]$, ~\cite{tarres2014online} proves almost sure convergence for regularized least-mean-squares without averaging at rate $O(n^{-\frac{1+\theta}{2+\theta}})$. The dependence  in~$\theta$ is similar to what we obtain. Moreover, under assumption~\ref{hypo:?}, we recover the same $O(1/\sqrt{n})$ convergence rate as~\cite{bhandari2018finite} (see  Prop.~\ref{easythm} stated in App.~\ref{sec:proofs_iid}). Finally, our bounds have a polynomial dependence in the horizon~$1/(1-\gamma)$ of the~MRP.

\section{Stochastic TD with Markovian sampling} \label{sec:markov}

We now consider the truly \textit{online} TD algorithm, where the samples are produced by a Markov chain. In particular, there is now a correlation between the current samples $(x_n, x_n')$ and the previous iterate~$V_{n-1}$. To control it, we assume that the Markov chain mixes at uniform geometric rate:
\phantomsection
\label{hypo:???} 
\begin{align}
\textbf{(A3)} \qquad\qquad \exists m>0,~ \mu \in(0, 1) \text{~ s.t. }\sup_{x \in \cX} d_{TV} \left(\mathbb{P}(x_n \in \cdot | x_0 = x), p \right) \leq m \mu^n, \qquad\qquad ~ \label{eqn:mixing}
\end{align}
where $d_{TV}$ denotes the total variation distance. This is always verified for irreducible, aperiodic finite Markov chains~\cite{levin2017markov}. We give an example of continuous-state Markov chain with geometric mixing in Sec.~\ref{sec:numerics}. Furthermore, following~\cite{bhandari2018finite}, in our analysis we need to control the magnitude of the iterates almost surely. To do so, we add a projection step at each TD iteration:
\begin{align}
    V_{n} = \Pi_B [ V_{n-1} + \rho_n ((A_n - \lambda I) V_{n-1} + b_n) ], \label{eqn:TD_proj}
\end{align}
where $\Pi_B$ is the projection on the $\cH$ ball of radius $B>0$. If $\|V^*_\lambda\|_\cH \leq B$, the convergence of the method is preserved. In the following theorem, we consider two regimes with different rates of convergence. In the first one, we assume like~\cite{bhandari2018finite} that we are given an oracle~$B$ upper-bounding $\|V^*_\lambda\|_\cH$, with $B$ independent of~$\lambda$. In the second one, we use the bound of Prop.~\ref{prop3}, but this will affect the convergence rate since in this case~$B = O(1/\lambda)$.

\begin{theorem} \label{thm2} Assuming~\textup{\ref{hypo:??}} and that the samples are produced by a Markov chain with uniform geometric mixing~\hyperref[hypo:???]{\textbf{\textup{(A3)}}}, the projected TD iterations~(\ref{eqn:TD_proj}) are such that:
\setlist{nolistsep} \begin{enumerate}[label=(\roman*), noitemsep]
\item 
\label{hypo:thm2_i}
Using $\lambda = n^{-\frac{1}{2+\theta}}$, a constant step size $\rho = \frac{\log n}{2\lambda n}$, and using a projection radius~$B$ independent of $\lambda$ provided by an oracle and such that $\|V^*_\lambda\|_\cH \leq B$, then:
\begin{align}
    \Ee \left[ \| V^{\textit{(e)}}_n - V^* \|_{L^2(p)}^2 \right] \leq O \Big( \frac{(\log n)^2 n^{-\frac{1+\theta}{2+\theta}} }{\log(1/\mu)} \Big).
\end{align}
\item
\label{hypo:thm2_ii}
Using $\lambda = n^{-\frac{1}{4+\theta}}$, $\rho = \frac{\log n}{2\lambda n}$, 
and the projection radius~$B=\sqrt{M_\cH} \|r\|_{L^2(p)}/\lambda$, then:
\begin{align}
    \Ee \left[ \| V^{\textit{(e)}}_n - V^* \|_{L^2(p)}^2 \right] \leq O \Big( \frac{(\log n)^2 n^{-\frac{1+\theta}{4+\theta}} }{\log(1/\mu)} \Big),
\end{align}
with $V^{\textit{(e)}}_n = \sum_{k=1}^n  (1-2\rho \lambda)^{n-k}  V_{k-1} /\sum_{j=1}^n (1-2\rho \lambda)^{n-j}.$
\end{enumerate}
\end{theorem}

When an oracle is given for~$B$ (i.e., setting \ref{hypo:thm2_i}), we recover the same rate as \textit{i.i.d.}~sampling, up to a multiplicative factor $\log(n)/\log(1/\mu)$ which represents the mixing time of the Markov chain. If no oracle is provided (i.e., setting \ref{hypo:thm2_ii}), the convergence will be slower because the bound~$B$ is of order~$1/\lambda$. Note that the slight changes in the definitions of~$\rho$, $\lambda$, $V^{(e)}$, and the absence of constraint on~$\lambda$, as compared to Thm.~\ref{thm1}, are implied by the boundedness of the iterates. Following a similar study for  SGD~\cite{nagaraj2020least}, we might compare these rates to those of a naive algorithm which we call ``$\tau$-Skip-TD'', for some $\tau\geq1$, where only one every~$\tau$ samples from the Markov chain is used to make TD updates:
\begin{align}
    V_{n} = \Pi_B [ V_{n-1} + \rho_n ((A_{n\tau} - \lambda I) V_{n-1} + b_{n\tau}) ], \label{eqn:skip_TD}
\end{align}
For $\tau$ large enough, of the order of the mixing time of the Markov chain, the new sample $(x_{n\tau}, x'_{n\tau})$ is almost independent from the past ones $(x_{k\tau}, x'_{k\tau})_{k<n}$. Of course, since we need to generate~$\tau$ times more samples to make a step, we must look at the distance of $V_{n/\tau}$ to the optimum. Such convergence rates for $\tau$-Skip-TD are derived in App.~\ref{sec:proofs_markov}, Cor.~\ref{cor1}. In setting \textit{(i)}, they are similar to Theorem~\ref{thm2} up to a $\log(n)$ factor. This suggests that making updates at each sample of the Markov chain is not more efficient than $\tau$-Skip-TD for large $\tau$, at least in our theoretical analysis. In practice,  using all samples seems slightly better, especially for a slowly mixing Markov chain (see App.\ref{subsec:exps}).  In setting \textit{(ii)}, we obtain a rate for Skip-TD whose leading term does not depend on $\log(1/\mu)$ -- which only appears in higher order terms -- suggesting that the rate and  parameters of Thm.~\ref{thm2}, setting~\textit{(ii)} might be suboptimal.

\section{Experiment on artificial data} \label{sec:numerics}

\textbf{Building a value function.} We build a toy model for which the main parameters can be computed in closed form. We consider the dynamics on the circle $\cX = [0, 1]$ defined by: with probability~$\e$, $x_{n+1}\sim \cU([0, 1])$, and with probability~$1-\e$, $x_{n+1} = x_n$. Because the Markov kernel is symmetric, the invariant distribution is $p=\cU([0, 1])$. In particular, the mixing parameter can be bounded explicitly with $m = 1$ and $\mu = 1-\e$ (see App.~\ref{subsec:mixing}). Also, simple computations show that $V^*$ is an affine transform of $r$: $V^*(x) = a r(x) + b$, with $a=(1-\gamma(1-\e))^{-1}$ and $b= -a \int_0^1 r(u) du$. Hence we can build a~$V^*$ with a given regularity by choosing an appropriate reward with the same regularity. We consider two rewards: $r_\text{abs}(x) := 2 | x-1/2 |$ and  $r_\text{cos}(x) := (1+\cos(2\pi x))/2$.

\textbf{Kernels on the torus.} 
We consider the RKHS of splines on the circle~\cite{wahba1990spline} of regularity~$s \in \Nn^*$, denoted by~$H_\text{per}^s$. It is a Sobolev space equipped with the following norm:
\begin{align}
    \|f\|^2_{H^s_\text{per}} = \left( \int_0^1 f(x) dx \right)^2 + \frac{1}{(2\pi)^{2s}} \int_0^1 |f^{(s)}(x)|^2 dx. \label{eqn:sob_norm}
\end{align}
Its corresponding reproducing kernel $K_s$ is a translation-invariant kernel defined by:
\begin{align}
    K_s(x, y) = 1 + (-1)^{s-1} \frac{(2 \pi)^{2s}}{(2s)!}B_{2s}(\{x-y \}), \label{eqn:kspline}
\end{align}
where $\{x\} := x - \lfloor x \rfloor$ and $B_j$ is the $j$-th  Bernoulli polynomial~\cite{olver2010nist}.
Let us recall that the Fourier series expansion on the torus of a 1-periodic function $f\in L^2(p)$ is: $f(x) = \sum_{\omega \in \Zz} e^{2i \omega \pi x} \hat f_\omega$, with $\hat f_\omega := \int_0^1 f(x) e^{-2i \omega \pi x} dx$, for $\omega \in \Zz$. The kernel $K_s$ has an embedding in the space of Fourier coefficients $\Phi(x) = (\sqrt{c_\omega}e^{2i \omega \pi x})^\top_{m \in \Zz}$ with $c_\omega := |\omega|^{-2s}$ if $\omega \neq 0$ and $c_0 := 1$. Using Parseval's theorem in Eqn.~(\ref{eqn:sob_norm}), one can compute the norm of $f$ from its Fourier coefficients $\| f\|^2_{H^s_\text{per}} = \sum_{\omega \in \Zz} |\hat f_\omega|^2/c_\omega$. The operators $\Sigma$ and $\Sigma_1$ can be represented as countably infinite-dimensional matrices $\Sigma = \text{diag}(c)$ and $\Sigma_1 = (1-\e) \Sigma + \e \sqrt{c} (\sqrt{c})^\top$. Hence the source condition states that $|\hat f_0 |^2 + \sum_{\omega \neq 0} |\omega|^{2s(1+\theta)} | \hat f_\omega |^2 < \infty$. In particular, it holds if $f \in H_\text{per}^{s'}$, for any $s' \geq s(1+\theta)$. In our example, we consider two Sobolev spaces $H_\text{per}^1$ and $H_\text{per}^2$, and our two example functions have Fourier coefficients $ ({\hat r_{\text{abs}}})_\omega  = \frac{1-(-1)^\omega}{\pi^2 \omega^2}$ for $\omega \neq 0$, and  $({\hat r_{\text{cos}}})_\omega = 0$ for $|\omega |>1$. The largest~$\theta \in [0, 1]$ such that the source condition holds are indicated in the first row of Tab.~\ref{table_rates}.

\textbf{Results.} We run TD on  functions $r_\text{abs}$ and $r_\text{cos}$, with kernels $K_1$ and $K_2$. We use parameters~$\lambda$ and~$\rho$ and exponential averaging as prescribed in Thm.~\ref{thm1}~\textit{(b)}. Each experiment is repeated~10 times and we record the mean $\pm$ one standard deviation.  The implementation is based on a finite dimensional representation of the iterates $(V_k)_{k \leq n}$ in $\Rr^n$ (see further details in App.~\ref{subsec:implem}). This implies computing the kernel matrix in $O(n^2)$ operations. To accelerate this computation when the eigenvalues decrease fast, we approximate it with the incomplete Cholesky decomposition~\cite{bach2002kernel}. In Tab.~\ref{table_rates}, we set $\e=0.8$, $\gamma=0.5$ and report the observed convergence rates \textit{v.s.}~the ones expected by Thm.~\ref{thm2}, which are fairly consistent. In Fig.~\ref{fig1}, we show the respective effects of varying~$\e$ and~$\gamma$. Larger values of~$\e$ or~$\gamma$ make the problem more difficult and slow down convergence, presumably in the constants without affecting the rates, as predicted by Thm.~\ref{thm2}. Additional experiments are provided in App.~\ref{subsec:exps}.

\begin{table}[ht]
  \centering
  \caption{Predicted and observed convergence rates with different reward functions and kernels. \label{table_rates}}
\begin{tabular}{llllll}
\toprule
&\multicolumn{2}{c}{Sobolev kernel $s=1$}& ~ &\multicolumn{2}{c}{Sobolev kernel $s=2$}\\
\cmidrule(l){2-6}
& \hfil$r=r_\text{abs}$ & \hfil$r=r_\text{cos}$ &~& \hfil$r=r_\text{abs}$ & \hfil$r=r_\text{cos}$\\
\midrule
\hfil Maximal $\theta$ & \hfil$1/2$ & \hfil$1$ &~&\hfil$-1/4$& \hfil$1$\\
\hfil Predicted rate & \hfil $- 0.6$ &\hfil $- 0.67$&~& \hfil $- 0.43$  & \hfil $-0.67$\\
\hfil Observed rate & \hfil $-0.72$ &\hfil $-0.64$ &~& \hfil $-0.58$ & \hfil $-0.64$\\
\bottomrule
\end{tabular}
\end{table}

\begin{figure}[ht]
    \centering
    \includegraphics[width=0.47\linewidth,trim={0 0.2cm 0.2cm 0},clip]{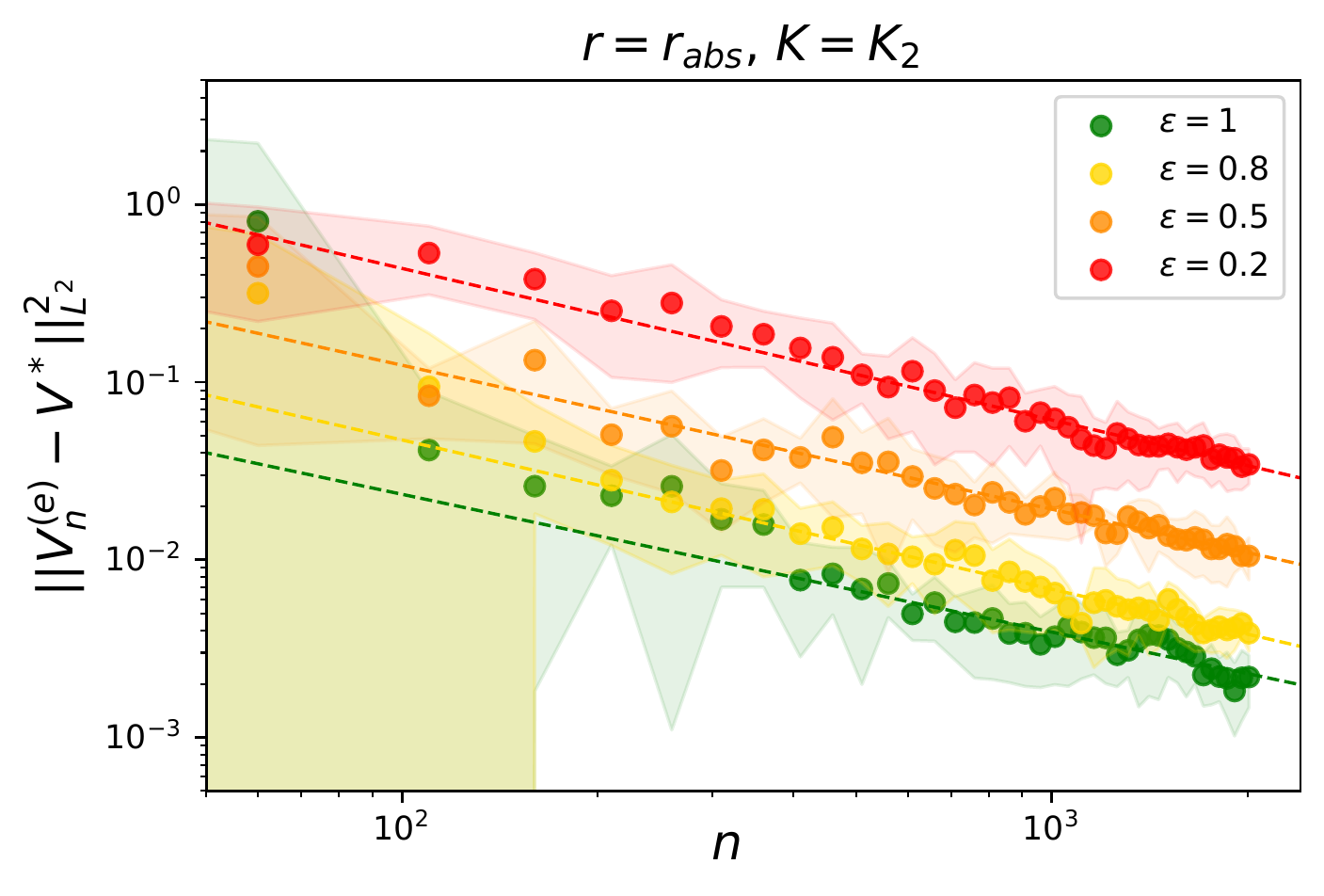} \hspace{0.03\linewidth}
    \includegraphics[width=0.47\linewidth,trim={0 0.2cm 0.2cm 0},clip]{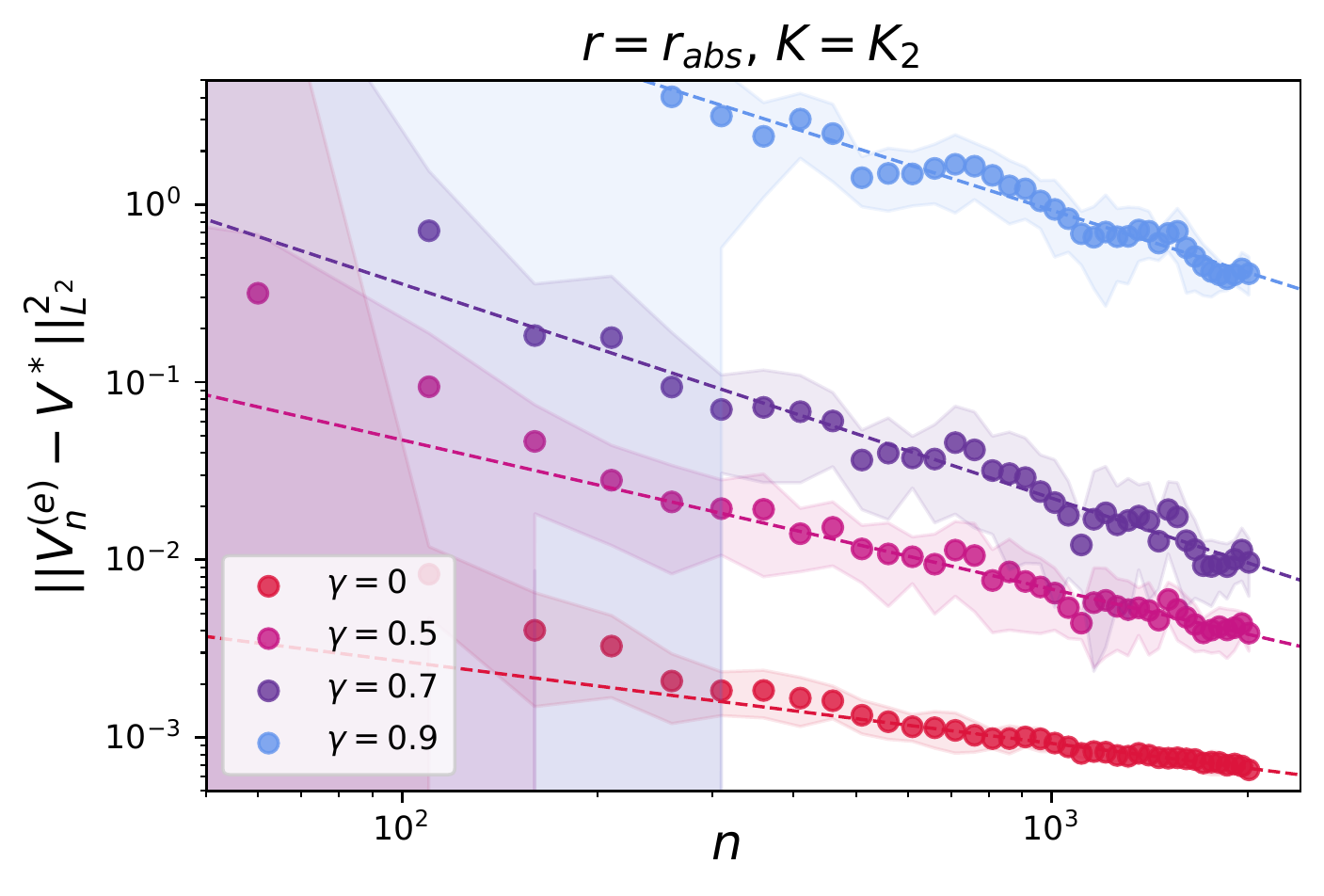}
    \caption{Respective effects of varying $\e$ (for $\gamma=0.5$ fixed) and $\gamma$ (for $\e=0.8$ fixed).}
    \label{fig1}
\end{figure}

\section{Conclusion} \label{sec:conclusion}
We have provided convergence rates for  the regularized non-parametric TD algorithm in the \textit{i.i.d.}~and Markovian sampling settings. The rates depend on a source condition that quantifies the relative regularity of the optimal value function to the RKHS. They are compatible with our empirical observations on a one-dimensional MRP, but we have not proved optimality of such rates. Interesting directions include the extension to the TD($\lambda$) algorithm, which we believe can be achieved with similar tools, as well as more challenging extensions to control counterparts of TD (Q-learning, SARSA,...) for which the policy is optimized.

\section*{Acknowledgements}

This work was supported by the Direction Générale de l’Armement, and by the French government under management of Agence Nationale de la Recherche as part of the “Investissements d’avenir” program, reference ANR-19-P3IA-0001 (PRAIRIE 3IA Institute). We also acknowledge support from the European Research Council (grant SEQUOIA 724063).

\bibliographystyle{abbrv}
{\bibliography{biblio}}

\begin{thebibliography}{10}

\bibitem{asmussen2003applied}
S.~Asmussen.
\newblock {\em Applied Probability and Queues}, volume~2.
\newblock Springer, 2003.

\bibitem{bach2022information}
F.~Bach.
\newblock Information theory with kernel methods.
\newblock {\em arXiv preprint arXiv:2202.08545}, 2022.

\bibitem{bach2002kernel}
F.~Bach and M.~I. Jordan.
\newblock Kernel independent component analysis.
\newblock {\em Journal of Machine Learning Research}, 3(Jul):1--48, 2002.

\bibitem{baird1995residual}
L.~Baird.
\newblock Residual algorithms: Reinforcement learning with function
  approximation.
\newblock In {\em Machine Learning Proceedings 1995}, pages 30--37. Elsevier,
  1995.

\bibitem{baker1973joint}
C.~R. Baker.
\newblock Joint measures and cross-covariance operators.
\newblock {\em Transactions of the American Mathematical Society},
  186:273--289, 1973.

\bibitem{barreto2011reinforcement}
A.~Barreto, D.~Precup, and J.~Pineau.
\newblock Reinforcement learning using kernel-based stochastic factorization.
\newblock {\em Advances in Neural Information Processing Systems}, 24, 2011.

\bibitem{barreto2016practical}
A.~M. Barreto, D.~Precup, and J.~Pineau.
\newblock Practical kernel-based reinforcement learning.
\newblock {\em The Journal of Machine Learning Research}, 17(1):2372--2441,
  2016.

\bibitem{benveniste2012adaptive}
A.~Benveniste, M.~M{\'e}tivier, and P.~Priouret.
\newblock {\em Adaptive Algorithms and Stochastic Approximations}, volume~22.
\newblock Springer Science \& Business Media, 1990.

\bibitem{berthier2020tight}
R.~Berthier, F.~Bach, and P.~Gaillard.
\newblock Tight nonparametric convergence rates for stochastic gradient descent
  under the noiseless linear model.
\newblock {\em Advances in Neural Information Processing Systems},
  33:2576--2586, 2020.

\bibitem{bhandari2018finite}
J.~Bhandari, D.~Russo, and R.~Singal.
\newblock A finite time analysis of temporal difference learning with linear
  function approximation.
\newblock In {\em Conference on Learning Theory}, pages 1691--1692, 2018.

\bibitem{bhat2012non}
N.~Bhat, V.~Farias, and C.~C. Moallemi.
\newblock Non-parametric approximate dynamic programming via the kernel method.
\newblock {\em Advances in Neural Information Processing Systems}, 25, 2012.

\bibitem{bhatia2013matrix}
R.~Bhatia.
\newblock {\em Matrix Analysis}, volume 169.
\newblock Springer Science \& Business Media, 2013.

\bibitem{borkar2000ode}
V.~S. Borkar and S.~P. Meyn.
\newblock The {ODE} method for convergence of stochastic approximation and
  reinforcement learning.
\newblock {\em SIAM Journal on Control and Optimization}, 38(2):447--469, 2000.

\bibitem{bottou2018optimization}
L.~Bottou, F.~E. Curtis, and J.~Nocedal.
\newblock Optimization methods for large-scale machine learning.
\newblock {\em SIAM Review}, 60(2):223--311, 2018.

\bibitem{boyan1994generalization}
J.~Boyan and A.~Moore.
\newblock Generalization in reinforcement learning: Safely approximating the
  value function.
\newblock {\em Advances in Neural Information Processing Systems}, 7, 1994.

\bibitem{bradtke1996linear}
S.~J. Bradtke and A.~G. Barto.
\newblock Linear least-squares algorithms for temporal difference learning.
\newblock {\em Machine Learning}, 22(1):33--57, 1996.

\bibitem{caponnetto2007optimal}
A.~Caponnetto and E.~De~Vito.
\newblock Optimal rates for the regularized least-squares algorithm.
\newblock {\em Foundations of Computational Mathematics}, 7(3):331--368, 2007.

\bibitem{cheney2001analysis}
E.~W. Cheney.
\newblock {\em Analysis for Applied Mathematics}, volume~1.
\newblock Springer, 2001.

\bibitem{cristianini2004kernel}
N.~Cristianini and J.~Shawe-Taylor.
\newblock {\em Kernel Methods for Pattern Analysis}, volume 173.
\newblock Cambridge University Press, 2004.

\bibitem{cucker2002mathematical}
F.~Cucker and S.~Smale.
\newblock On the mathematical foundations of learning.
\newblock {\em Bulletin of the American Mathematical Society}, 39(1):1--49,
  2002.

\bibitem{cucker2007learning}
F.~Cucker and D.~X. Zhou.
\newblock {\em Learning Theory: an Approximation Theory Viewpoint}, volume~24.
\newblock Cambridge University Press, 2007.

\bibitem{dai2017learning}
B.~Dai, N.~He, Y.~Pan, B.~Boots, and L.~Song.
\newblock Learning from conditional distributions via dual embeddings.
\newblock In {\em Artificial Intelligence and Statistics}, pages 1458--1467,
  2017.

\bibitem{dalal2017finite}
G.~Dalal, B.~Sz{\"o}r{\'e}nyi, G.~Thoppe, and S.~Mannor.
\newblock Finite sample analyses for {TD}(0) with function approximation.
\newblock {\em AAAI'18/IAAI'18/EAAI'18}, 2018.

\bibitem{dayan1992convergence}
P.~Dayan.
\newblock The convergence of {TD}($\lambda$) for general $\lambda$.
\newblock {\em Machine Learning}, 8(3):341--362, 1992.

\bibitem{defossez2017adabatch}
A.~D{\'e}fossez and F.~Bach.
\newblock Adabatch: Efficient gradient aggregation rules for sequential and
  parallel stochastic gradient methods.
\newblock {\em arXiv preprint arXiv:1711.01761}, 2017.

\bibitem{dietterich2001batch}
T.~Dietterich and X.~Wang.
\newblock Batch value function approximation via support vectors.
\newblock {\em Advances in Neural Information Processing Systems}, 14, 2001.

\bibitem{dieuleveut2017stochastic}
A.~Dieuleveut.
\newblock {\em Stochastic Approximation in {H}ilbert Spaces}.
\newblock PhD thesis, Paris Sciences et Lettres (ComUE), 2017.

\bibitem{dieuleveut2016nonparametric}
A.~Dieuleveut and F.~Bach.
\newblock Nonparametric stochastic approximation with large step-sizes.
\newblock {\em The Annals of Statistics}, 44(4):1363--1399, 2016.

\bibitem{domingues2021kernel}
O.~D. Domingues, P.~M{\'e}nard, M.~Pirotta, E.~Kaufmann, and M.~Valko.
\newblock Kernel-based reinforcement learning: A finite-time analysis.
\newblock In {\em International Conference on Machine Learning}, pages
  2783--2792, 2021.

\bibitem{duan2021optimal}
Y.~Duan, M.~Wang, and M.~J. Wainwright.
\newblock Optimal policy evaluation using kernel-based temporal difference
  methods.
\newblock {\em arXiv preprint arXiv:2109.12002}, 2021.

\bibitem{durrett2019probability}
R.~Durrett.
\newblock {\em Probability: Theory and Examples}, volume~49.
\newblock Cambridge University Press, 2019.

\bibitem{farahmand2016regularized}
A.-M. Farahmand, M.~Ghavamzadeh, C.~Szepesv{\'a}ri, and S.~Mannor.
\newblock Regularized policy iteration with nonparametric function spaces.
\newblock {\em The Journal of Machine Learning Research}, 17(1):4809--4874,
  2016.

\bibitem{fukumizu2004dimensionality}
K.~Fukumizu, F.~R. Bach, and M.~I. Jordan.
\newblock Dimensionality reduction for supervised learning with reproducing
  kernel {H}ilbert spaces.
\newblock {\em Journal of Machine Learning Research}, 5(Jan):73--99, 2004.

\bibitem{grunewalder2012modelling}
S.~Grünewälder, G.~Lever, L.~Baldassarre, M.~Pontil, and A.~Gretton.
\newblock Modelling transition dynamics in {MDP}s with {RKHS} embeddings.
\newblock In {\em International Conference on Machine Learning}, 2012.

\bibitem{halko2011finding}
N.~Halko, P.-G. Martinsson, and J.~A. Tropp.
\newblock Finding structure with randomness: Probabilistic algorithms for
  constructing approximate matrix decompositions.
\newblock {\em SIAM review}, 53(2):217--288, 2011.

\bibitem{jaakkola1993convergence}
T.~Jaakkola, M.~Jordan, and S.~Singh.
\newblock Convergence of stochastic iterative dynamic programming algorithms.
\newblock {\em Advances in Neural Information Processing Systems}, 6, 1993.

\bibitem{klenke2013probability}
A.~Klenke.
\newblock {\em Probability Theory: A Comprehensive Course}.
\newblock Springer Science \& Business Media, 2013.

\bibitem{koppel2020policy}
A.~Koppel, G.~Warnell, E.~Stump, P.~Stone, and A.~Ribeiro.
\newblock Policy evaluation in continuous {MDP}s with efficient kernelized
  gradient temporal difference.
\newblock {\em IEEE Transactions on Automatic Control}, 66(4):1856--1863, 2020.

\bibitem{korda2015td}
N.~Korda and P.~La.
\newblock On {TD}(0) with function approximation: Concentration bounds and a
  centered variant with exponential convergence.
\newblock In {\em International Conference on Machine Learning}, pages
  626--634, 2015.

\bibitem{lakshminarayanan2018linear}
C.~Lakshminarayanan and C.~Szepesvari.
\newblock Linear stochastic approximation: How far does constant step-size and
  iterate averaging go?
\newblock In {\em International Conference on Artificial Intelligence and
  Statistics}, pages 1347--1355, 2018.

\bibitem{levin2017markov}
D.~A. Levin and Y.~Peres.
\newblock {\em Markov Chains and Mixing Times}, volume 107.
\newblock American Mathematical Society, 2017.

\bibitem{long20212}
J.~Long, J.~Han, and W.~E.
\newblock An ${L}^{2}$ analysis of reinforcement learning in high dimensions
  with kernel and neural network approximation.
\newblock {\em arXiv preprint arXiv:2104.07794}, 2021.

\bibitem{micchelli2006universal}
C.~A. Micchelli, Y.~Xu, and H.~Zhang.
\newblock Universal kernels.
\newblock {\em Journal of Machine Learning Research}, 7(12), 2006.

\bibitem{mou2020optimal}
W.~Mou, A.~Pananjady, and M.~J. Wainwright.
\newblock Optimal oracle inequalities for solving projected fixed-point
  equations.
\newblock {\em arXiv preprint arXiv:2012.05299}, 2020.

\bibitem{nagaraj2020least}
D.~Nagaraj, X.~Wu, G.~Bresler, P.~Jain, and P.~Netrapalli.
\newblock Least squares regression with {M}arkovian data: Fundamental limits
  and algorithms.
\newblock {\em Advances in Neural Information Processing Systems}, 2020.

\bibitem{novak2018reproducing}
E.~Novak, M.~Ullrich, H.~Wo{\'z}niakowski, and S.~Zhang.
\newblock Reproducing kernels of {S}obolev spaces on $\mathbb{R}^d$ and
  applications to embedding constants and tractability.
\newblock {\em Analysis and Applications}, 16(05):693--715, 2018.

\bibitem{olver2010nist}
F.~W.~J. Olver, D.~W. Lozier, R.~F. Boisvert, and C.~W. Clark.
\newblock {\em NIST Handbook of Mathematical Functions}.
\newblock Cambridge University Press, 2010.

\bibitem{ormoneit2002kernel}
D.~Ormoneit and {\'S}.~Sen.
\newblock Kernel-based reinforcement learning.
\newblock {\em Machine Learning}, 49(2):161--178, 2002.

\bibitem{pillaud2018exponential}
L.~Pillaud-Vivien, A.~Rudi, and F.~Bach.
\newblock Exponential convergence of testing error for stochastic gradient
  methods.
\newblock In {\em Conference on Learning Theory}, pages 250--296, 2018.

\bibitem{polyak1992acceleration}
B.~T. Polyak and A.~B. Juditsky.
\newblock Acceleration of stochastic approximation by averaging.
\newblock {\em SIAM Journal on Control and Optimization}, 30(4):838--855, 1992.

\bibitem{reiss2012course}
R.-D. Reiss.
\newblock {\em A Course on Point Processes}.
\newblock Springer Science \& Business Media, 2012.

\bibitem{rudin}
W.~Rudin.
\newblock {\em Real and Complex Analysis, 3rd Ed.}
\newblock McGraw-Hill, Inc., USA, 1987.

\bibitem{schapire1996worst}
R.~E. Schapire and M.~K. Warmuth.
\newblock On the worst-case analysis of temporal-difference learning
  algorithms.
\newblock {\em Machine Learning}, 22(1):95--121, 1996.

\bibitem{slotine1991applied}
J.-J.~E. Slotine and W.~Li.
\newblock {\em Applied Nonlinear Control}, volume 199.
\newblock Prentice Hall Englewood Cliffs, NJ, 1991.

\bibitem{srikant2019finite}
R.~Srikant and L.~Ying.
\newblock Finite-time error bounds for linear stochastic approximation and {TD}
  learning.
\newblock In {\em Conference on Learning Theory}, pages 2803--2830, 2019.

\bibitem{steinwart2001influence}
I.~Steinwart.
\newblock On the influence of the kernel on the consistency of support vector
  machines.
\newblock {\em Journal of Machine Learning Research}, 2(Nov):67--93, 2001.

\bibitem{sutton1988learning}
R.~S. Sutton.
\newblock Learning to predict by the methods of temporal differences.
\newblock {\em Machine Learning}, 3(1):9--44, 1988.

\bibitem{sutton2015introduction}
R.~S. Sutton.
\newblock Introduction to reinforcement learning with function approximation.
\newblock In {\em Tutorial at the Conference on Neural Information Processing
  Systems}, page~33, 2015.

\bibitem{sutton2018reinforcement}
R.~S. Sutton and A.~G. Barto.
\newblock {\em Reinforcement Learning: An Introduction}.
\newblock MIT press, 2018.

\bibitem{tarres2014online}
P.~Tarres and Y.~Yao.
\newblock Online learning as stochastic approximation of regularization paths:
  Optimality and almost-sure convergence.
\newblock {\em IEEE Transactions on Information Theory}, 60(9):5716--5735,
  2014.

\bibitem{tsitsiklis1997analysis}
J.~N. Tsitsiklis and B.~Van~Roy.
\newblock An analysis of temporal-difference learning with function
  approximation.
\newblock {\em IEEE Transactions on Automatic Control}, 42(5):674--690, 1997.

\bibitem{wahba1990spline}
G.~Wahba.
\newblock {\em Spline Models for Observational Data}.
\newblock CBMS-NSF Regional Conference Series in Applied Mathematics. Society
  for Industrial and Applied Mathematics, 1990.

\bibitem{weidmann2012linear}
J.~Weidmann.
\newblock {\em Linear Operators in {H}ilbert Spaces}, volume~68.
\newblock Springer Science \& Business Media, 2012.

\bibitem{xu2020reanalysis}
T.~Xu, Z.~Wang, Y.~Zhou, and Y.~Liang.
\newblock Reanalysis of variance reduced temporal difference learning.
\newblock {\em arXiv preprint arXiv:2001.01898}, 2020.

\bibitem{yu2010error}
H.~Yu and D.~P. Bertsekas.
\newblock Error bounds for approximations from projected linear equations.
\newblock {\em Mathematics of Operations Research}, 35(2):306--329, 2010.

\end{thebibliography}


\appendix


\section{Proofs and intermediate results} \label{sec:proofs}

\subsection{Problem formulation and generic results} \label{sec:proofs_gen}

\begin{proof}[Proof of Lemma~\ref{prop2}]
Let $V \in L^2(p)$. Then:
\begin{align*}
    \| PV \|_{L^2(p)}^2 &= \int_\cX  (\Ee_{x' \sim \kappa(x, \cdot)} V(x') )^2 p(dx) \\
    &\leq \int_\cX  \Ee_{x' \sim \kappa(x, \cdot)} [V(x')^2]  p(dx) \\
    &= \int_{\cX} \left( \int_\cX V(x')^2  \kappa(x, dx') \right)  p(dx)\\
    &= \int_\cX V(x')^2 \left(\int_\cX \kappa(x, dx') p(dx) \right)  \\
    &= \int_\cX V(x')^2 p(dx')  \\
    & = \| V\|_{L^2(p)}^2.
\end{align*}
The second line is an application of Jensen's inequality, with equality if $\forall x, V(x')|x$ is  constant almost surely (a.s.). The fourth line is an application of Fubini-Tonelli's theorem. The fifth line results from the stationarity of~$p$ with respect to~$\kappa$, and~$\kappa(\cdot, dx')$ being $\cA$-measurable.
\end{proof}

\subsection{Analysis of a continuous-time version of the population TD algorithm} \label{sec:proofs_pop}

Proposition~\ref{prop3} is a consequence of the following Lemma~\ref{lemma1}:
\begin{lemma} \label{lemma1}
For $\lambda >0$, the operator $\Sigma + \lambda I - \gamma \Sigma_1 : \cH \rightarrow \cH$ is bijective, and the operator norm of its inverse is bounded as follows:
$$\| (\Sigma + \lambda I - \gamma \Sigma_1 )^{-1} \|_{\rm op} \leq  \frac{1}{\lambda}. $$
\end{lemma}

\begin{proof}[Proof of Lemma~\ref{lemma1}] 
From Lemma~\ref{prop1}, there exists $\tilde \Sigma_1$ with $\| \tilde \Sigma_1 \|_{\rm op} \leq 1$ such that $\Sigma_1 = \Sigma^{1/2} \tilde \Sigma_1 \Sigma^{1/2}$.

For $\lambda >0$, $\Sigma + \lambda I \succ 0$, hence we have the decomposition:
\begin{align}
    \Sigma + \lambda I -\gamma \Sigma_1  = (\Sigma + \lambda I)^{1/2} \left[ I - \gamma (\Sigma + \lambda I)^{-1/2} \Sigma^{1/2} \tilde \Sigma_1 \Sigma^{1/2} (\Sigma + \lambda I)^{-1/2}  \right] (\Sigma + \lambda I )^{1/2}. \label{eqn:factor}
\end{align}
Since the operator norm is an induced norm:
\begin{align*}
     &\| (\Sigma + \lambda I)^{-1/2} \Sigma^{1/2} \tilde \Sigma_1 \Sigma^{1/2} (\Sigma + \lambda I)^{-1/2} \|_{\rm op} \\
     &\qquad\qquad\qquad \leq \| (\Sigma + \lambda I)^{-1/2}\Sigma^{1/2} \|_{\rm op}  \cdot \| \tilde \Sigma_1  \|_{\rm op} \cdot \| \Sigma^{1/2} (\Sigma + \lambda I)^{-1/2} \|_{\rm op}.
\end{align*}
Furthermore, $\Sigma^{1/2} (\Sigma + \lambda I)^{-1/2} \preceq I$, hence:
$$ \|\gamma (\Sigma + \lambda I)^{-1/2} \Sigma^{1/2} \tilde \Sigma_1 \Sigma^{1/2} (\Sigma + \lambda I)^{-1/2} \|_{\rm op} \leq \gamma < 1.$$
We can then apply Theorem 5.11 from~\cite{weidmann2012linear}, showing that the term inside the brackets in Eqn.~(\ref{eqn:factor}) is invertible, with inverse equal to:
\begin{align}
    \sum_{k=0}^{+\infty} \gamma^k [(\Sigma + \lambda I)^{-1/2} \Sigma^{1/2} \tilde \Sigma_1 \Sigma^{1/2} (\Sigma + \lambda I)^{-1/2}]^k. \label{eqn:gamma_sum}
\end{align} 
Hence, $\Sigma + \lambda I -\gamma \Sigma_1 $ is invertible, with inverse equal to:
$$ (\Sigma + \lambda I)^{-1/2} \left[I - \gamma (\Sigma + \lambda I)^{-1/2} \Sigma^{1/2} \tilde \Sigma_1 \Sigma^{1/2} (\Sigma + \lambda I)^{-1/2}  \right]^{-1} (\Sigma + \lambda I )^{-1/2}. $$

We will now upper-bound the operator norm of $(\gamma \Sigma_1 - \Sigma - \lambda I)^{-1}$. Let us take $f,g\in \cH$
    such that $g=(\lambda I + \Sigma -\gamma \Sigma_1)f$
    and $\|g\|_\cH=1$,
    we get
    \begin{align*}
        1 &= \| (\lambda I + \Sigma  -\gamma \Sigma_1) f \|_\cH^2         \\
        &= \lambda^2 \|f \|^2_\cH + 2 \lambda \langle f, \Sigma f \rangle_\cH - \lambda \gamma \langle f, (\Sigma_1 + \Sigma^*_1) f \rangle_\cH + \| (\Sigma - \gamma \Sigma_1) f\|_\cH^2\\
        & \geq \lambda^2 \|f \|^2_\cH + 2 \lambda \langle f, \Sigma f \rangle_\cH - \lambda \gamma \langle f, (\Sigma_1 + \Sigma^*_1) f \rangle_\cH.
    \end{align*}
    
    Moreover, we have:
    \begin{align*}
        \langle f, \Sigma_1 f \rangle_\cH & = \Ee_q [f(x) f(x')] \\
        & \leq  \Ee_q \left[\frac{f(x)^2}{2} + \frac{f(x')^2}{2} \right] \\
        & = \Ee_{x \sim p} \left[\frac{f(x)^2}{2} \right] + \Ee_{x' \sim p} \left[\frac{f(x')^2}{2} \right] \\
        & = \langle f, \Sigma f \rangle_\cH,
    \end{align*}
    because~$p$ is an invariant distribution. Similarly, $$\langle f, \Sigma_1^* f \rangle_\cH =\langle \Sigma_1 f,  f \rangle_\cH = \langle f, \Sigma_1 f \rangle_\cH  \leq \langle f, \Sigma f \rangle_\cH.$$
    
    Consequently, since $\gamma \leq 1$, 
    we get $1\geq
    \lambda^2 \|f\|^2 = 
    \lambda^2 \| (\lambda I
        + \Sigma 
        -\gamma \Sigma_1)^{-1}g \|^2_\cH$.
    We conclude by using the definition of
    the operator norm, \textit{i.e.},
    \begin{equation*}
        \| (\lambda I
        + \Sigma 
        -\gamma \Sigma_1)^{-1}\|_{\rm op}
    =\sup_{\|g\|_\cH=1}\| (\lambda I
        + \Sigma 
        -\gamma \Sigma_1)^{-1}g\|_\cH \leq 1/\lambda.
    \end{equation*}
\end{proof}

\begin{proof}[Proof of Proposition~\ref{prop3}] 
Consider the fixed point equation~(\ref{eqn:fixed}). Since $\lambda >0$, it is equivalent to:
$$V = \frac{1}{\lambda} \left[ \Sigma r + \gamma \Sigma_1 V - \Sigma V \right] . $$
As a consequence, any solution of this equation is in $\cH$. Using Lemma~\ref{lemma1}, it is unique and such that:
$$V = (\gamma \Sigma_1 - \Sigma - \lambda I)^{-1} \Sigma r. $$
\end{proof}

\begin{proof}[Proof of Proposition~\ref{prop4}] The fixed point equations verified by $V^*_\lambda$ and $V^*$ are respectively:
\begin{align}
    \Sigma r + (\gamma \Sigma_1 - \Sigma - \lambda I) V^*_\lambda &= 0. \label{eqn:flb} \\
    \Sigma r + (\gamma \Sigma_1 - \Sigma - \lambda I) V^* &= - \lambda V^* \label{eqn:f0}
\end{align}
Let $\bar V^* := \Sigma^{1/2} V^* $, $\bar V^*_\lambda := \Sigma^{1/2} V^*_\lambda $, and $\bar r := \Sigma^{1/2} r$. Then $\bar V^*$, $\bar V^*_\lambda$ and $\bar r$ are all in $\cH$. Using Lemma~\ref{prop1}, there exists $\tilde \Sigma_1 : \cH \rightarrow \cH$ with $\| \tilde \Sigma_1 \|_{\rm op} \leq 1$ such that $\Sigma_1 = \Sigma^{1/2} \tilde \Sigma_1 \Sigma^{1/2}$. Because of assumption~\ref{hypo:??}, this equality holds on $\overline{\cH} = L^2(p)$.  In particular, $\Sigma^{1/2} \Sigma_1 V^* = \Sigma \tilde \Sigma_1 \bar V^*$.

Left multiplying Eqns.~(\ref{eqn:flb}) and~(\ref{eqn:f0}) by $\Sigma^{1/2}$, we get:
\begin{align}
    \Sigma \bar r + (\gamma \Sigma \tilde \Sigma_1 - \Sigma - \lambda I) \bar V^*_\lambda &= 0. \label{eqn:flbs} \\
    \Sigma \bar r + (\gamma \Sigma \tilde \Sigma_1  - \Sigma - \lambda I) \bar V^* &= - \lambda \bar V^* \label{eqn:f0s}
\end{align}
Subtracting Eqns.~(\ref{eqn:flbs}) and~(\ref{eqn:f0s}), we get:
\begin{align}
   (\Sigma + \lambda I - \gamma \Sigma \tilde \Sigma_1) ( \bar V_\lambda^* - \bar V^*) = - \lambda  \bar V^*. \label{eqn:diff}
\end{align}
Since $\Sigma + \lambda I \succ 0$, then:
\begin{align*}
   (I - \gamma (\Sigma + \lambda I)^{-1} \Sigma \tilde \Sigma_1) ( \bar V_\lambda^* - \bar V^*) = - \lambda  (\Sigma + \lambda I)^{-1} \bar V^*.
\end{align*}
Let $\tilde \Sigma_{1,\lambda} := (\Sigma + \lambda I)^{-1} \Sigma \tilde \Sigma_1$. Since $(\Sigma + \lambda I)^{-1} \Sigma \preceq I$,  we know that $\| \gamma \tilde \Sigma_{1,\lambda} \|_{\rm op} \leq \gamma < 1$. Hence $(I - \gamma \tilde \Sigma_{1,\lambda})$ is invertible and:
\begin{align*}
    \bar V^*_\lambda - \bar V^* &= -\lambda (I - \gamma \tilde \Sigma_{1,\lambda})^{-1} (\Sigma + \lambda I)^{-1} \bar V^* \\
    & = -\lambda \sum_{k=0}^{+\infty} \gamma^k \tilde \Sigma_{1,\lambda}^k (\Sigma + \lambda I)^{-1} \Sigma^{1/2} V^* .
\end{align*}
Taking the $\cH$-norm on both sides, and using the isometry property~(\ref{eqn:isometry}), valid on $\overline{\cH} =L^2(p)$ since we are using a universal kernel:
\begin{align}
    \|\Sigma^{1/2} (V^*_\lambda -  V^*)\|_\cH & \leq \lambda \sum_{k=0}^{+\infty} \gamma^k \| \tilde \Sigma_{1,\lambda}^k (\Sigma + \lambda I)^{-1} \Sigma^{1/2} V^*\|_\cH \\
    \|V^*_\lambda -  V^*\|_{L^2(p)} &\leq \lambda \sum_{k=0}^{+\infty} \gamma^k \| (\Sigma + \lambda I)^{-1} \Sigma^{1/2} V^*\|_\cH \\
    &=  \frac{\lambda}{1-\gamma} \| (\Sigma + \lambda I)^{-1} \Sigma^{1/2} V^*\|_\cH \label{eqn:sum}.
\end{align}
Assuming that~$V^*$ verifies the source condition with constant~$\theta$, the norm on the right-hand side can be bounded as follows:
\begin{align*}
    \| (\Sigma + \lambda I)^{-1}\Sigma^{1/2} V^* \|_\cH &= \| (\Sigma + \lambda I)^{-1}  \Sigma^{(1+\theta)/2}\Sigma^{-\theta/2} V^* \|_\cH \\
    &=  \| (\Sigma + \lambda I)^{(\theta-1)/2} (\Sigma + \lambda I)^{-(1+\theta)/2}  \Sigma^{(1+\theta)/2}\Sigma^{-\theta/2} V^* \|_\cH\\
    & \leq \lambda^{(\theta - 1)/2} \| (\Sigma + \lambda I)^{-(1+\theta)/2}  \Sigma^{(1+\theta)/2}\Sigma^{-\theta/2} V^* \|_\cH,
\end{align*}
because $ 0 \prec (\Sigma + \lambda I)^{(\theta - 1)/2} \preceq \lambda^{(\theta - 1)/2} I$, since $(\theta - 1)/2 \leq 0$. Also, since $(1+\theta)/2 \geq 0$, we have: $ (\Sigma + \lambda I)^{-(1+\theta)/2} \Sigma^{(1+\theta)/2} \preceq I$,  hence:
\begin{align}
    \| (\Sigma + \lambda I)^{-1} \Sigma^{1/2} V^* \|_\cH & \leq  \lambda^{(\theta-1)/2} \|  \Sigma^{-\theta/2} V^* \|_\cH. \label{eqn:k0}
\end{align}
Combining Eqns.~(\ref{eqn:sum}) and~(\ref{eqn:k0}), we can then conclude that:
\begin{align*}
    \|V_\lambda^* - V^*\|_{L^2(p)} &  \leq \frac{\lambda^{\frac{1+\theta}{2}}}{1-\gamma}  \|\Sigma^{-\theta/2} V^* \|_\cH.
\end{align*}
\end{proof}

\begin{corollary} \label{cor_mercer} Assume that $K$ is a universal Mercer kernel, and that $V^* \in L^2(p)$ (which holds as soon as $r \in L^2(p)$, see Sec.~\ref{subseq}), then:
$$\|V_\lambda^* - V^*\|_{L^2(p)} \xrightarrow[\lambda \to 0^+]{} 0.$$
\end{corollary}

\begin{proof}[Proof of Corollary~\ref{cor_mercer}] 
We can reproduce the beginning of the proof of Prop.~\ref{prop4}, until Eqn.~(\ref{eqn:sum}):
\begin{align*}
    \|V^*_\lambda -  V^*\|_{L^2(p)} &\leq   \frac{\lambda}{1-\gamma} \| (\Sigma + \lambda I)^{-1} \Sigma^{1/2} V^*\|_\cH.
\end{align*}

Using the isometry property~(\ref{eqn:isometry})  because $K$ is a universal kernel:
\begin{align*}
    \|V^*_\lambda -  V^*\|_{L^2(p)} &\leq   \frac{\lambda}{1-\gamma} \| (\Sigma + \lambda I)^{-1} V^*\|_{L^2(p)}.
\end{align*}

Because $K$ is a Mercer kernel, there exists a sequence $(\psi_n)_{n \geq 1}$ in $L^2(p)$ which is an orthonormal eigenbasis of $\overline{\cH} = L^2(p)$ (because $K$ is universal) for the $L^2(p)$ inner product, with strictly positive eigenvalues $(\lambda_n)_{n \geq 1}$, ordered in decreasing order, such that~\cite{dieuleveut2016nonparametric}:
$$\forall n\geq 1, \quad \Sigma \psi_n = \lambda_n \psi_n.  $$
Then, since $V^* = \sum_{n \geq 1} \langle V^*, \psi_n \rangle_{L^2(p)} \psi_n$:
\begin{align*}
     \|V^*_\lambda -  V^*\|_{L^2(p)}^2 &\leq   \frac{\lambda^2}{(1-\gamma)^2} \| (\Sigma + \lambda I)^{-1} V^*\|^2_{L^2(p)} \\
     &= \frac{1}{(1-\gamma)^2} \sum_{n \geq 1} \frac{\lambda^2}{(\lambda + \lambda_n)^2} \langle V^*, \psi_n \rangle_{L^2(p)}^2 .
\end{align*}
For $\lambda >0$, the series on the right-hand side is dominated by $$\sum_{n \geq 1} \langle V^*, \psi_n \rangle_{L^2(p)}^2 = \|V\|_{L^2(p)}^2 < \infty, $$
and for each $n \geq 1$:
$$ \frac{\lambda^2}{(\lambda + \lambda_n)^2} \langle V^*, \psi_n \rangle_{L^2(p)}^2 \xrightarrow[\lambda \to 0^+]{}0,$$
because each $\lambda_n$ is strictly positive. Then by  Lebesgue's dominated convergence theorem~\cite{rudin}:
$$\|V_\lambda^* - V^*\|^2_{L^2(p)} \xrightarrow[\lambda \to 0^+]{} 0.$$
\end{proof}

\begin{proof}[Proof of Lemma~\ref{prop5}]
For $\beta = 0$, and $\lambda >0$, $V_t - V_\lambda^* \in \Sigma^{0/2}(\cH) = \cH$ is always true as proved in Prop.~\ref{prop3}, hence $W^0(t)$ is finite for all $t \geq 0$. Similarly, $W^1(t)$ is finite for all $t \geq 0$ because $V_t$ and $V^*_\lambda \in L^2(p)$.
\begin{align*}
    \frac{dW^0(t)}{dt} &= 2 \langle  V_t - V_\lambda^*,  \frac{dV_t}{dt} \rangle_\cH \\
    &=2 \langle  V_t - V_\lambda^*,  (A-\lambda I)V_t  + b \rangle_\cH \\
    &= 2 \langle V_t - V_\lambda^*,  (\gamma \Sigma_1 - \Sigma -\lambda I)V_t) + \Sigma r \rangle_\cH.
    \end{align*}
We remind that $V^*_\lambda$ is a solution of Eqn.~(\ref{eqn:fixed}). Then:
        \begin{align*}
    \frac{dW^0}{dt} &= 2\langle V_t - V_\lambda^*, (\gamma \Sigma_1 - \Sigma - \lambda I)(V_t - V_\lambda^*) \rangle_\cH \\
    &=2 \gamma \langle V_t - V_\lambda^*, \Sigma_1  (V_t - V_\lambda^*) \rangle_\cH -2 \lambda \langle  V_t - V_\lambda^*, V_t - V^*_\lambda \rangle_\cH - 2\langle V_t - V_\lambda^* , \Sigma(V_t - V_\lambda^*) \rangle_\cH \\
    &=2 \gamma \langle V_t - V_\lambda^*, \Sigma P (V_t - V_\lambda^*) \rangle_\cH - 2\lambda W^0(t) - 2W^{-1}(t) \\
    &=2 \gamma \langle \Sigma^{1/2} (V_t - V_\lambda^*), \Sigma^{1/2} P (V_t - V^*) \rangle_\cH - 2\lambda W^0(t) - 2W^{-1}(t), 
\end{align*}
where the third line results from Eqn.~(\ref{eqn:bellman_sigma}). Using Cauchy-Schwarz inequality for $\langle \cdot, \cdot \rangle_\cH$, the first term is bounded by:
\begin{align*} 2\gamma \langle \Sigma^{1/2} (V_t - V_\lambda^*), \Sigma^{1/2} P (V_t - V^*) \rangle_\cH  &\leq 2 \gamma \| \Sigma^{1/2} (V_t - V^*_\lambda ) \|_\cH \cdot \|\Sigma^{1/2} P (V_t - V_\lambda^*)  \|_\cH \\
& =2 \gamma \sqrt{W^{-1}(t)} \cdot \| P (V_t - V_\lambda^*)  \|_{L^2(p)} \\
& \leq 2 \gamma \sqrt{W^{-1}(t)} \cdot \| V_t - V_\lambda^* \|_{L^2(p)} \\
&= 2 \gamma W^{-1}(t),
\end{align*} 
where we have used successively Eqn.~(\ref{eqn:isometry}) (on an element of~$\cH$) and Lemma~\ref{prop2}.

Finally, we get:
\begin{align*}
    \frac{dW^0(t)}{dt} &\leq  2 \gamma W^{-1}(t) - 2 \lambda W^0(t) - 2 W^{-1}(t),
\end{align*}
where all of the above quantities are finite. 
\end{proof}

\begin{proof}[Proof of Proposition~\ref{prop6}]
We treat separately the two sets of assumptions.

 $\bullet$~ Under assumption~\ref{hypo:?}, we define the sequence of Polyak-Ruppert averaged iterates: $$\overline V_t := \frac{1}{t}\int_0^t V(s) ds, ~~\text{for~} t \geq 0.$$
 Lemma~\ref{prop5} can be easily adapted to the case where $\lambda = 0$, $\Sigma \succ 0$ and $V^* \in \cH$. The proof is the same, and all quantities are finite because $\|V^*\|_\cH$ is finite. Then we get:
$$\frac{d \|V_t - V^* \|^2_\cH}{dt} \leq - 2(1- \gamma) \|V_t - V^* \|^2_{L^2(p)} . $$
Let $T > 0$. Integrating between 0 and $T$ and dividing by $T$:
$$\frac{ W^0(T) -  W^0(0)}{T} \leq - 2(1- \gamma) \frac{1}{T}\int_0^T \|V_t - V^* \|^2_{L^2(p)} dt. $$
$$\frac{1}{T}\int_0^T \|V_t - V^* \|^2_{L^2(p)} dt \leq  \frac{1/2}{1-\gamma} \frac{W^0(0) - W^0(T)}{T} \leq \frac{1/2}{1-\gamma} \frac{W^0(0)}{T}. $$
Using Jensen's inequality:
$$\|\overline V_T - V^* \|^2_{L^2(p)} \leq \frac{1}{T}\int_0^T \|V_t - V^* \|^2_{L^2(p)} dt, $$
and then:
$$ \|\overline V_T - V^* \|^2_{L^2(p)} \leq \frac{1}{2(1-\gamma)} \frac{\|V^*\|_\cH^2}{T}. $$

$\bullet$~ Under assumption~\ref{hypo:??}, Lemma~\ref{prop5} gives:
\begin{align*}
    \frac{d \|V_t - V^*_\lambda \|^2_\cH}{dt} &\leq - 2 (1- \gamma) \|V_t - V^*_\lambda \|^2_{L^2(p)} - 2 \lambda \|V_t - V^*_\lambda \|_\cH^2 \\
    &\leq -2 \lambda \|V_t - V^*_\lambda \|_\cH^2.
\end{align*}
Using Grönwall's lemma, we directly get linear convergence of $V_t$ to $V^*_\lambda$ in $\cH$ norm:
$$ \|V_t - V_\lambda^* \|^2_\cH \leq \|V^*_\lambda\|^2_\cH e^{-2 t \lambda}.$$
\end{proof}

\subsection{Stochastic TD with \textit{i.i.d.}~sampling} \label{sec:proofs_iid}

First, we need to state a technical lemma which will be used several times:
\begin{lemma} \label{lemma2}
For any fixed $V \in L^2(p)$, and $n \geq 1$:
$$\Ee_q \| A_n V \|_\cH^2 \leq 2 M_\cH (1+\gamma^2) \|\Sigma^{1/2} V \|_\cH^2. $$
\end{lemma}

\begin{proof}[Proof of Lemma~\ref{lemma2}]
\begin{align*}
    \Ee_q \| A_n V \|_\cH^2 &= \Ee_q \| (\gamma \Phi(x_n) \otimes \Phi(x_n') - \Phi(x_n) \otimes \Phi(x_n)) V \|_\cH^2 \\
    & = \Ee_q \| \Phi(x_n) \otimes (\gamma \Phi(x_n') - \Phi(x_n)) V \|_\cH^2 \\
    & = \Ee_q \| \Phi(x_n) \|_\cH^2 | \langle V,  \gamma \Phi(x_n') - \Phi(x_n)  \rangle_\cH |^2 \\
    & \leq 2M_\cH (\gamma^2 \Ee_q[ \langle V, \Phi(x_n')\rangle_\cH^2 ] +\Ee_q[ \langle V, \Phi(x_n) \rangle_\cH^2] ) .
\end{align*}
Since the expectation is according to the  distribution $q$, the two random variables inside the expectations have the same marginal distribution~$p$, and their expectation is equal to:
\begin{align*}
    \Ee_{x \sim p}[ \langle V, \Phi(x) \rangle_\cH^2] &=\Ee_{x \sim p}[ \langle V, \langle V, \Phi(x) \rangle_\cH \Phi(x) \rangle_\cH] \\
    & =\Ee_{x \sim p}[ \langle V, \Phi(x) \otimes \Phi(x) V \rangle_\cH] \\
    & =\langle V, \Sigma V \rangle_\cH = \| \Sigma^{1/2} V \|_\cH^2,
\end{align*}
which yields the result.
\end{proof}

We now derive the stochastic equivalent of the Descent Lemma~\ref{prop5}.

\begin{lemma} \label{lemma3} Let $\sigma^2 := 10 M_\cH \|r\|^2_{L^2(p)} +  \left(\frac{8  (1+\gamma^2) }{(1-\gamma)^2}  + 16 (1+\gamma^2) \right) M_\cH  \| V^* \|^2_{L^2(p)}$.

Then for $n \geq 1$:
\begin{align*}
    \Ee W^0_n &\leq(1-2\rho_n \lambda + 2 \rho_n^2 \lambda^2) \Ee W^0_{n-1} - \left(2 \rho_n (1 - \gamma) - 8\rho_n^2(1+\gamma^2)  M_\cH \right) \Ee W^{-1}_{n-1} +4 \rho_n^2\sigma^2 .
\end{align*}
In particular, for $\rho_n \leq \min \left\{ \frac{1}{2\lambda}, \frac{1-\gamma}{8 M_\cH (1+\gamma^2)} =: \bar \rho \right\}$:
\begin{align*}
    \Ee W^0_n &\leq(1-\rho_n \lambda ) \Ee W^0_{n-1} -  \rho_n (1 - \gamma) \Ee W^{-1}_{n-1} +4 \rho_n^2  \sigma^2.
\end{align*}
\end{lemma}

\begin{proof}[Proof of Lemma~\ref{lemma3}]
We have the following decomposition, almost surely:
\begin{align*}
    W_n^0 & = \langle V_n - V^*_\lambda,V_n - V^*_\lambda \rangle_\cH \\
    &= \langle V_{n-1} + \rho_n ((A_n - \lambda I)V_{n-1}+b_n) - V^*_\lambda,V_{n-1} + \rho_n ((A_n - \lambda I)V_{n-1}+b_n) - V^*_\lambda \rangle_\cH \\
    &= \langle V_{n-1} - V^*_\lambda,V_{n-1} - V^*_\lambda \rangle_\cH + 2 \rho_n \langle V_{n-1} - V^*_\lambda, (A_n - \lambda I)V_{n-1} + b_n \rangle_\cH \\&\qquad\qquad\qquad\qquad\qquad\qquad + \rho_n^2 \| (A_n - \lambda I) V_{n-1} + b_n \|_\cH^2.
\end{align*}
Let $z_i := (x_i, x_i')$, for $i \geq 1$. The $z_i$ are \textit{i.i.d.}~with probability distribution~$q$. Taking the expectation with respect to the filtration $\cF_n := \sigma(z_1,...,z_n)$, we get three terms:
\begin{align*}
    \Ee W_n^0 & = \Ee W_{n-1}^0 + 2 \rho_n \Ee \left[ \langle V_{n-1} - V^*_\lambda, (A_n - \lambda I)V_{n-1} + b_n \rangle_\cH \right] \\&\qquad\qquad\qquad\qquad\qquad\qquad + \rho_n^2 \Ee \left[ \| (A_n - \lambda I) V_{n-1} + b_n \|^2_\cH \right].
\end{align*}

$\bullet$ We first consider the inner product:
\begin{align*}
    \Ee \left[ \langle V_{n-1} - V^*_\lambda, (A_n - \lambda I)V_{n-1} + b_n \rangle_\cH \right] &= \Ee \left[ \Ee \left[ \langle V_{n-1} - V^*_\lambda, (A_n - \lambda I)V_{n-1} + b_n \rangle_\cH | \cF_{n-1}\right] \right] \\
    &= \Ee \left[ \langle V_{n-1} - V^*_\lambda, (A - \lambda I)V_{n-1} + b \rangle_\cH  \right] \\
    &\leq - (1-\gamma ) \Ee W_{n-1}^{-1} - \lambda \Ee W^0_{n-1},
\end{align*}
where we used the expectation of Lemma~\ref{prop5} on the last line:
\begin{align*}
    \langle V - V^*_\lambda, (A - \lambda I)V + b \rangle_\cH \leq - (1- \gamma) \|V - V^*_\lambda \|^2_{L^2(p)} - \lambda \|V - V^*_\lambda \|_\cH^2.
\end{align*}

$\bullet$ Now we need to upper-bound the final variance term:
\begin{align*}
    \Ee \left[ \| (A_n - \lambda I) V_{n-1} + b_n \|^2_\cH \right] & \leq 2  \Ee \left[ \| \lambda (V_{n-1} - V^*_\lambda) \|_\cH^2 \right] + 2  \Ee \left[ \| A_n V_{n-1} + b_n - \lambda V^*_\lambda \|_\cH^2 \right] \\
    & \leq 2 \lambda^2 \Ee W^0_{n-1} + 4 \Ee \left[ \| A_n(V_{n-1} - V^*_\lambda)  \|_\cH^2 \right] \\
    &\qquad\qquad\qquad\qquad + 4 \Ee \left[ \|(A_n - \lambda I) V^*_\lambda + b_n \|_\cH^2 \right] \\
    & \leq 2 \lambda^2 \Ee W^0_{n-1} + 4 \Ee \left[ \Ee \left[ \| A_n(V_{n-1} - V^*_\lambda) \|_\cH^2 | \cF_{n-1} \right] \right] \\
    &\qquad\qquad\qquad\qquad+ 4 \Ee \left[ \|(A_n - \lambda I) V^*_\lambda + b_n \|_\cH^2 \right] \\
    & \leq 2 \lambda^2 \Ee W^0_{n-1} + 8M_\cH (1+\gamma^2) \Ee W^{-1}_{n-1} \\
    & \qquad\qquad\qquad\qquad + 4 \Ee \left[ \|(A_n - \lambda I) V^*_\lambda + b_n \|_\cH^2 \right],
    \end{align*}
the last inequality being an application of Lemma~\ref{lemma3} to $V_{n-1} - V^*_\lambda$,  deterministic given $\cF_{n-1}$.

Next, we are going to show that the remaining  variance term $\Ee \left[ \|(A_n - \lambda I) V_{\lambda}^* + b_n \|^2_\cH \right]$ is bounded and give an explicit upper-bound~$\sigma^2$. This is the variance of the updates at the optimum:
    \begin{align*}
    \Ee \left[ \|(A_n - \lambda I) V_{\lambda}^* + b_n \|^2_\cH \right] & \leq  2 \lambda^2 \|V_\lambda^*\|^2_\cH + 2 \Ee \left[ \|A_n V^*_\lambda + b_n \|^2_\cH \right] \\
    & \leq  2 M_\cH \|r \|^2_{L^2(p)} + 2 \Ee \left[ \|A_n V^*_\lambda + b_n \|_\cH^2 \right],
    \end{align*}
using Prop.~\ref{prop3}. Then:
    \begin{align*}
    2 \Ee \left[ \|A_n V^*_\lambda + b_n \|_\cH^2 \right]&\leq  4 \Ee \left[ \|A_n (V^*_\lambda-V^*) \|^2_\cH \right] + 4 \Ee \left[ \|A_n V^* + b_n \|^2_\cH \right] \\
    &\leq  8 M_\cH (1+\gamma^2) \|\Sigma^{1/2} (V^*_\lambda - V^*) \|_\cH^2  + 4 \Ee \left[ \|A_n V^* + b_n \|^2_\cH \right],
    \end{align*}
applying Lemma \ref{lemma3} to $V^*_\lambda - V^*$. Then, using Prop.~\ref{prop4} with $\theta=-1$ (which always holds):
    \begin{align*}
    2 \Ee \left[ \|A_n V^*_\lambda + b_n \|_\cH^2 \right]& \leq \frac{8 M_\cH (1+\gamma^2) \|V^*\|_{L^2(p)}^2}{(1-\gamma)^2}   + 4 \Ee \left[ \|A_n V^* + b_n \|^2_\cH \right]   \\
    & \leq  \frac{8 M_\cH (1+\gamma^2) \|V^*\|_{L^2(p)}^2}{(1-\gamma)^2}   + 8 \Ee \left[ \|A_n V^* \|_\cH^2 \right] +   8 \Ee \left[ \|b_n \|_\cH^2 \right] \\
    & \leq \frac{8 M_\cH (1+\gamma^2) \|V^*\|_{L^2(p)}^2}{(1-\gamma)^2} + 16 M_\cH (1+\gamma^2) \| V^* \|^2_{L^2(p)}  \\
   & \qquad\qquad\qquad\qquad\qquad\qquad+ 8 M_\cH \|r\|^2_{L^2(p)},
\end{align*}
where we have used again Lemma~\ref{lemma3} applied to $V^*$, and the fact that:
\begin{align*}
    \Ee [\|b_n\|_\cH^2] =\Ee [r(x_n)^2 \|\Phi(x_n)\|_\cH^2] \leq M_\cH \Ee_p [r(x_n)^2 ] =  M_\cH \| r\|^2_{L^2(p)}.
\end{align*}

Hence the variance $\Ee \left[ \|(A_n - \lambda I) V_{\lambda}^* + b_n \|^2_\cH \right]$ is finally bounded by: $$\sigma^2 := 10 M_\cH \|r\|^2_{L^2(p)} +  \left(\frac{8  (1+\gamma^2) }{(1-\gamma)^2}  + 16 (1+\gamma^2) \right) M_\cH  \| V^* \|^2_{L^2(p)}.$$ Back to the main term, we get:
\begin{align*}
    \Ee \left[ \| (A_n - \lambda I) V_{n-1} + b_n \|^2_\cH \right] & \leq 2 \lambda^2 \Ee W^0_{n-1} + 8M_\cH (1+\gamma^2) \Ee W^{-1}_{n-1} + 4 \sigma^2. 
    \end{align*}
Then, we get the result:
\begin{align*} 
   \Ee W^0_n &\leq \Ee W_{n-1}^0 - 2 \rho_n (1-\gamma) \Ee W^{-1}_{n-1} - 2 \rho_n \lambda \Ee W^0_{n-1} \\ &\qquad + 2 \rho_n^2   \lambda^2 \Ee W^0_{n-1} + 8 \rho_n^2 M_\cH (1+\gamma^2) \Ee W^{-1}_{n-1} + 4 \rho^2_{n-1} \sigma^2.
\end{align*}
\end{proof}

\begin{proposition} \label{easythm} Under assumption \textup{\ref{hypo:?}}, there exists an $n_0 >0$ such that, when using a constant step size $\rho = 1/\sqrt{n}$ and $\lambda = 0$,  the Polyak-Ruppert averaged iterates $\overline V_n$, for $n \geq n_0$ verify:
$$\Ee \| \overline V_n - V^* \|_{L^2(p)}^2 \leq O(1/\sqrt{n}).$$
\end{proposition}

\begin{proof}[Proof of Proposition~\ref{easythm}]
We use Lemma~\ref{lemma3} with $\lambda = 0$: if  $\rho_k \leq \bar \rho$, 
$$\Ee W_k^0 \leq \Ee W_{k-1}^0 - \rho_k (1-\gamma) \Ee W_{k-1}^{-1} + 4 \rho_k^2 \sigma^2. $$
We use a constant step size $\rho$. Then:
\begin{align*}
    \Ee W_{k-1}^{-1} &\leq  \frac{\Ee W^0_{k-1} - \Ee W^0_k}{\rho(1-\gamma)} +  \frac{4 \rho\sigma^2}{1-\gamma} .
\end{align*}
Summing for $k$ and dividing by $n$, we get a telescoping sum:
\begin{align*}
    \frac{1}{n}\sum_{k=1}^n \Ee W_{k-1}^{-1} &\leq  \frac{\Ee W^0_{0} - \Ee W^0_{n}}{n \rho(1-\gamma)} +  \frac{4 \rho\sigma^2}{1-\gamma} \leq  \frac{\Ee W^0_{0} }{n \rho(1-\gamma)} +  \frac{4 \rho\sigma^2}{1-\gamma}.
\end{align*}
Using Jensen's inequality:
$$ \Ee \| \overline V_n - V^* \|_{L^2(p)}^2 \leq \frac{\|V^*\|^2_\cH}{(1-\gamma) \rho n} + \frac{4 \rho\sigma^2}{1-\gamma}. $$

We choose a constant step size $\rho = 1/\sqrt{n}$. For $n \geq n_0 := 1/\bar \rho^2$, $\rho_n \leq \bar \rho$, hence the application of Lemma~\ref{lemma3} is valid and we get a rate:
$$\Ee \| \overline V_n - V^* \|_{L^2(p)}^2 \leq O(1/\sqrt{n}).$$
\end{proof}

\begin{proof}[Proof of Theorem~\ref{thm1}] For each case, we first assume that $\lambda$ and $\rho_n$ are such that the conditions of Lemma~\ref{lemma3} are satisfied. Then we pick particular choices of~$\lambda$ and~$\rho_n$ to obtain the convergence rate, and check that the conditions are indeed satisfied.

(a) Let $\lambda >0$ and $\rho$ a constant step size such that $\rho \leq \bar \rho$ and $\rho \leq 1/(2\lambda)$. In this case, Lemma~\ref{lemma3} reads:
\begin{align*}
    \Ee W^0_n &\leq(1-\rho \lambda ) \Ee W^0_{n-1} -  \rho (1 - \gamma) \Ee W^{-1}_{n-1} +4 \rho^2  \sigma^2.
\end{align*}
In particular:
\begin{align} \Ee W_n^0 \leq (1 - \rho \lambda) \Ee W_{n-1}^0  + 4 \rho^2 \sigma^2. \label{eqn:fpi}
\end{align}
Removing the fixed point of this inequality~(\ref{eqn:fpi}) on both sides, we get:
\begin{align}
    \Ee W_n^0 - \frac{4 \rho \sigma^2}{\lambda} \leq (1 - \rho \lambda) \left( \Ee W_{n-1}^0 - \frac{4 \rho \sigma^2}{\lambda} \right). \label{eqn:geom} 
\end{align}
Since $\rho \lambda \leq 1/2$, this is a contracting geometric sequence and, applying~(\ref{eqn:geom}) recursively, we get:
\begin{align*}
    \Ee W_n^0 - \frac{4 \rho \sigma^2}{\lambda} & \leq (1 - \rho \lambda)^n \left( \Ee W_{0}^0 - \frac{4 \rho \sigma^2}{\lambda} \right) \\
    & \leq (1 - \rho \lambda)^n  \Ee W_{0}^0 .
\end{align*}
Finally, using Prop.~\ref{prop3}:
\begin{align}
    \Ee W_n^0  \leq \frac{4 \rho \sigma^2}{\lambda} + (1 - \rho \lambda)^n \frac{M_\cH \|r\|^2_{L^2(p)}}{\lambda^2 } . \label{eqn:resA}
\end{align}
We now consider specific choices of $\lambda$ and $\rho$. Let $\lambda = \lambda_0 n^{- \frac{1}{3+\theta}}$ and $\rho = \frac{\log n}{\lambda n}$, for some $\lambda_0$. Let us look at the conditions of Lemma~\ref{lemma3}:
\begin{itemize}
    \item $\rho \leq 1/(2\lambda)$ if and only if $\frac{\log n}{n} \leq 1/2$, which is true for all $n \geq 1$.
    \item $\rho \leq \bar \rho$ if and only if $(\log n) n^{\frac{1}{3+\theta} - 1}/\lambda_0 \leq \bar \rho$. Since $\theta > -1$, $\frac{1}{3+\theta} - 1 < -1/2$, hence $(\log n) n^{\frac{1}{3+\theta} - 1} / \bar \rho \rightarrow 0$. In particular it is bounded for all $n \geq 1$. Hence defining:
    $$\underline{\lambda}_\theta^{(0)} := \max \{(\log n) n^{\frac{1}{3+\theta} - 1} / \bar \rho ~|~ n \geq 1 \},$$ then for $\lambda_0 \geq \underline{\lambda}_\theta^{(0)}$, $\rho \leq \bar \rho$ is satisfied. Note that $\underline{\lambda}_\theta^{(0)}$ is independent of~$n$.
\end{itemize}
For this choice of $\lambda$ and $\rho$, we get:
$$\Ee W_n^0  \leq \frac{4 \sigma^2  \log n}{ \lambda_0^2 n^{1-\frac{2}{3+\theta}}} + \left(1 - \frac{\log n}{n} \right)^n \frac{M_\cH \|r\|^2_{L^2(p)}}{\lambda_0^2 n^{-\frac{2}{3+\theta}} } . $$

For $n \geq 1$, $\log\left(1 - \frac{\log n}{n} \right) \leq - \frac{\log n}{n}$, hence $\left(1 - \frac{\log n}{n} \right)^n \leq 1/n$ and:
$$\Ee W_n^0  \leq \frac{4 \sigma^2  (\log n) n^{-\frac{1+\theta}{3+\theta}} }{\lambda_0^2} +  \frac{M_\cH \|r\|^2_{L^2(p)}}{ \lambda_0^2 } n^{-\frac{1+\theta}{3+\theta}} . $$

We can then obtain convergence to $V^*$ at the same rate, using Prop.~\ref{prop4}:
\begin{align*}
    \Ee \|V_n - V^* \|_{L^2(p)}^2 &\leq 2 M_\cH \Ee \|V_n - V^*_\lambda \|^2_\cH + 2 \|V^*_\lambda - V^* \|^2_{L^2(p)} \\
    & \leq \frac{8 M_\cH \sigma^2  }{\lambda_0^2}(\log n) n^{-\frac{1+\theta}{3+\theta}} +  \frac{2M_\cH^2 \|r\|^2_{L^2(p)}}{ \lambda_0^2} n^{-\frac{1+\theta}{3+\theta}} \\
    &\qquad\qquad\qquad\qquad+ \frac{2  \| \Sigma^{-\theta/2} V^*\|_\cH^2 \lambda_0^{1+\theta}}{(1-\gamma)^2} n^{-\frac{1+\theta}{3+\theta}}.
\end{align*}

(b) Let $\lambda >0$ and $\rho$ a constant step size such that $\rho \leq \bar \rho$ and $\rho \leq 1/(2\lambda)$. In this case, Lemma~\ref{lemma3} reads, for each $k \in \{1, ..., n\}$:
\begin{align}
    \Ee W^0_k &\leq(1-\rho \lambda ) \Ee W^0_{k-1} -  \rho (1 - \gamma) \Ee W^{-1}_{k-1} +4 \rho^2  \sigma^2. \label{eqn:rec}
\end{align}
Using~(\ref{eqn:rec}) recursively, we obtain:
\begin{align*}
   \Ee W_n^0 &\leq (1 - \rho \lambda)^n \Ee W_0^0 - (1-\gamma) \rho \sum_{k=1}^n  (1-\rho \lambda)^{n-k}  \Ee W^{-1}_{k-1} + 4 \sigma^2 \rho^2 \sum_{k=1}^n  (1-\rho \lambda)^{n-k} .
\end{align*}
Re-arranging the terms, we get:
\begin{align*}
   \sum_{k=1}^n  (1-\rho \lambda)^{n-k}  \Ee W^{-1}_{k-1} &\leq \frac{(1 - \rho \lambda)^n}{\rho (1-\gamma)} \Ee W_0^0 - \frac{1}{\rho (1- \gamma)} \Ee W_n^0   + \frac{4 \sigma^2 \rho}{1-\gamma} \sum_{k=1}^n  (1-\rho \lambda)^{n-k} \\
   \sum_{k=1}^n  (1-\rho \lambda)^{n-k}  \Ee W^{-1}_{k-1} &\leq \frac{(1 - \rho \lambda)^n}{\rho (1-\gamma)} \frac{M_\cH \|r\|_{L^2(p)}^2}{\lambda^2} + \frac{4 \sigma^2 \rho}{1-\gamma} \sum_{k=1}^n  (1-\rho \lambda)^{n-k},
\end{align*}
using Prop.~\ref{prop3} on the last line.

Since $\sum_{k=1}^n  (1-\rho \lambda)^{n-k} = \frac{1 - (1 - \rho \lambda)^n}{\rho \lambda}$, we get:
\begin{align*}
    \frac{\sum_{k=1}^n  (1-\rho \lambda)^{n-k}  \Ee W^{-1}_{k-1}}{\sum_{k=1}^n (1-\rho \lambda)^{n-k}} &\leq \frac{ (1 - \rho \lambda)^n}{ 1 - (1 - \rho \lambda)^n} \frac{M_\cH \|r\|_{L^2(p)}^2}{\lambda(1-\gamma)} + \frac{4 \sigma^2 \rho}{1-\gamma}
\end{align*}

Using Jensen's inequality, we get:
\begin{align}
    \Ee \| V_n^{\textit{(e)}} - V^*_\lambda \|_{L^2(p)}^2 &\leq \frac{ (1 - \rho \lambda)^n}{ 1 - (1 - \rho \lambda)^n} \frac{M_\cH \|r\|_{L^2(p)}^2}{\lambda(1-\gamma)} + \frac{4 \sigma^2 \rho}{1-\gamma}, \label{eqn:ewa}
\end{align}
with $V_n^{\textit{(e)}} := \frac{\sum_{k=1}^n  (1-\rho \lambda)^{n-k}  V_{k-1}}{\sum_{k=1}^n (1-\rho \lambda)^{n-k}}$ the exponentially weighted average iterate.

Let $\lambda = \lambda_0 n^{- \frac{1}{2+\theta}}$, for some $\lambda_0 > 0$,  and $\rho = \frac{\log n}{\lambda n}$. The conditions of Lemma~\ref{lemma3} are:
\begin{itemize}
    \item $\rho \leq 1/(2\lambda)$ if and only if $\frac{\log n}{n} \leq 1/2$, which is true for all $n \geq 1$.
    \item $\rho \leq \bar \rho$ if and only if $(\log n) n^{\frac{1}{2+\theta} - 1}/\lambda_0 \leq \bar \rho$. Since $\theta > -1$, $\frac{1}{2+\theta} - 1 < 0$, hence $(\log n) n^{\frac{1}{2+\theta} - 1} / \bar \rho \rightarrow 0$. In particular it is bounded for all $n \geq 1$. Hence defining:
    $$\underline{\lambda}^{\textit{(e)}}_\theta := \max \{(\log n) n^{\frac{1}{2+\theta} - 1} / \bar \rho ~|~ n \geq 1 \},$$ then for $\lambda_0 \geq \underline{\lambda}^{\textit{(e)}}_\theta$, $\rho \leq \bar \rho$ is satisfied. Again, $\underline{\lambda}^{\textit{(e)}}_\theta$ is independent of~$n$.
\end{itemize}

For this choice of parameters, for $n \geq 2$: $$\left(1- \rho \lambda \right)^n = \left(1- \frac{\log n}{n} \right)^n = \exp \left(n \log \left(1- \frac{\log n}{n} \right) \right) \leq \exp\left(n \left(- \frac{\log n}{n} \right)\right) \leq \frac{1}{n} \leq \frac 12.$$

Hence:
\begin{align*}
    \Ee \|V_n^{\textit{(e)}} - V^*_\lambda \|^2_{L^2(p)} &\leq 2 (1 - \rho \lambda)^n \frac{n^{\frac{1}{2+\theta}} M_\cH \|r\|_{L^2(p)}^2}{\lambda_0(1-\gamma)} + \frac{4 \sigma^2 (\log n)n^{-\frac{1+\theta}{2+\theta}}}{\lambda_0(1-\gamma)} \\
    & \leq  \frac{2}{n} \cdot \frac{n^{\frac{1}{2+\theta}} M_\cH \|r\|_{L^2(p)}^2}{\lambda_0(1-\gamma)} + \frac{4 \sigma^2 (\log n)n^{-\frac{1+\theta}{2+\theta}}}{\lambda_0(1-\gamma)} \\
    &\leq   \frac{2 n^{-\frac{1+\theta}{2+\theta}} M_\cH \|r\|_{L^2(p)}^2}{\lambda_0(1-\gamma)} + \frac{4 \sigma^2 (\log n)n^{-\frac{1+\theta}{2+\theta}}}{\lambda_0(1-\gamma)} .
\end{align*}

We then obtain convergence to $V^*$ at the same rate, using Prop.~\ref{prop4}:
\begin{align*}
    \Ee \|V_n^{\textit{(e)}} - V^* \|_{L^2(p)}^2 &\leq 2  \Ee \|V_n^{\textit{(e)}} - V^*_\lambda \|^2_{L^2(p)} + 2 \|V^*_\lambda - V^* \|^2_{L^2(p)} \\
    & \leq  \frac{4  M_\cH \|r\|_{L^2(p)}^2}{\lambda_0(1-\gamma)} n^{-\frac{1+\theta}{2+\theta}} + \frac{8 \sigma^2 }{\lambda_0(1-\gamma)} (\log n)n^{-\frac{1+\theta}{2+\theta}}  \\
    &\qquad\qquad\qquad\qquad\qquad+ \frac{2  \| \Sigma^{-\theta/2} V^*\|_\cH^2 \lambda_0^{1+\theta}}{(1-\gamma)^2} n^{-\frac{1+\theta}{2+\theta}}.
\end{align*}

(c)  Let $n \geq 1$ and $\lambda >0$. We will consider a different step size schedule: first constant, then decreasing. For $k \in \{1, ..., 
\lfloor n/2\rfloor -1 \}$, set $\rho_k = \frac{2\log n}{\lambda n} =: \rho$. Then for $k \in \{\lfloor n/2\rfloor , ..., n \}$, set $\rho_k = \frac{1}{\lambda k}$.

$\bullet$~ We first look at the first $\lfloor n/2\rfloor - 1$ iterates.

Assume that $\lambda$ is chosen such that $\rho \leq \min \{1/(2\lambda) , \bar \rho \}$.  Under this condition, using the result~(\ref{eqn:resA}) that we derived above for setting~(a):
\begin{align}
    \Ee W_{\lfloor n/2\rfloor - 1}^0  \leq \frac{4 \rho \sigma^2}{\lambda} + (1 - \rho \lambda)^{\lfloor n/2\rfloor - 1} \frac{M_\cH \|r\|^2_{L^2(p)}}{\lambda^2 } . \label{eqn:1half}
\end{align}

$\bullet$~ Now for the next iterates, $\rho_k = \frac{1}{\lambda k}$. We also assume that $\lambda$ is chosen such that $\forall k \in \{ \lfloor n/2 \rfloor, ..., n \}$, $\rho_k \leq \min \{1/(2\lambda), \bar \rho \}$.  Under this condition, for $k \in \{ \lfloor n/2 \rfloor, ..., n \}$, Lemma~\ref{lemma3} reads:
\begin{align*}
    \Ee W^0_k &\leq(1-\rho_k \lambda ) \Ee W^0_{k-1} -  \rho_k (1 - \gamma) \Ee W^{-1}_{k-1} +4 \rho_k^2  \sigma^2.
\end{align*}
Re-arranging the terms:
\begin{align}
    \Ee W^{-1}_{k-1} & \leq  \frac{1}{1-\gamma} \left(\frac{1}{\rho_k}- \lambda \right) \Ee W^0_{k-1} - \frac{1}{1-\gamma} \frac{1}{\rho_k} \Ee W_k^0 + \frac{4 \sigma^2}{1-\gamma} \rho_k. \label{eqn:teles}
\end{align}
The step size is such that:
$$1/\rho_k - \lambda = \lambda k  - \lambda =\lambda (k-1)= 1/\rho_{k-1}, $$
where the very last equality only holds for $k \leq \lfloor n/2 \rfloor +1$ (because of overlapping notations).

Summing the above inequalities~(\ref{eqn:teles}) for $k \in \{ \lfloor n/2 \rfloor, ..., n \}$, we obtain a telescoping sum:
\begin{align*}
   \sum_{k=\lfloor n/2 \rfloor}^{n}  \Ee W^{-1}_{k-1} & \leq  \frac{1}{1-\gamma} \sum_{k=\lfloor n/2 \rfloor}^{n}   \left(\frac{\Ee W^0_{k-1}}{\rho_{k-1}} - \frac{\Ee W_k^0}{\rho_k}  \right) + \frac{4 \sigma^2}{1-\gamma} \sum_{k=\lfloor n/2 \rfloor}^{n} \rho_k \\
   & \leq \frac{1}{1-\gamma} \lambda (\lfloor n/2 \rfloor-1) \Ee W^0_{\lfloor n/2 \rfloor-1} + \frac{4 \sigma^2}{1-\gamma} \sum_{k=\lfloor n/2 \rfloor}^{n} \frac{1}{\lambda k} \\
   & \leq \frac{\lambda n}{2(1-\gamma)} \Ee W^0_{\lfloor n/2 \rfloor-1} + \frac{4 \sigma^2}{1-\gamma} \frac{1+\log n}{\lambda} .
 \end{align*}
 
 Using the result~(\ref{eqn:1half}) on the first half of the iterates, (for $n \geq 3$ so that $1+\log(n) \leq 2 \log n$):
 \begin{align*}
  \sum_{k=\lfloor n/2 \rfloor}^{n}  \Ee W^{-1}_{k-1} & \leq \frac{\lambda n}{2(1-\gamma)} \left[ \frac{4 \rho \sigma^2}{\lambda} + (1 - \rho \lambda)^{\lfloor n/2\rfloor - 1} \frac{M_\cH \|r\|^2_{L^2(p)}}{\lambda^2 }  \right] + \frac{8 \sigma^2}{1-\gamma} \frac{\log n}{\lambda} \\
  & \leq \frac{\lambda n}{2(1-\gamma)} \left[ \frac{8 (\log n) \sigma^2}{\lambda^2 n} + \left(1 - \frac{2 \log n}{n}\right)^{\lfloor n/2\rfloor - 1} \frac{M_\cH \|r\|^2_{L^2(p)}}{\lambda^2 }  \right] \\
  &\qquad\qquad\qquad\qquad\qquad\qquad\qquad\qquad\qquad\qquad+ \frac{8 \sigma^2}{1-\gamma} \frac{\log n}{\lambda} .
\end{align*}

Let us look at the central term:
\begin{align*}
    \left(1 - \frac{2 \log n}{n}\right)^{\lfloor n/2\rfloor - 1} &= \left(1 - \frac{2 \log n}{n}\right)^{n/2}\left(1 - \frac{2 \log n}{n}\right)^{\lfloor n/2\rfloor - 1 - n/2}
\end{align*}
Since $2 \log n / n \in [0, 1]$ for any $n \geq 1$,  and $\lfloor n/2 \rfloor - n/2 - 1 \geq -2$, we have:
\begin{align*}
    \left(1 - \frac{2 \log n}{n}\right)^{\lfloor n/2\rfloor - 1 - n/2} \leq \left(1 - \frac{2 \log n}{n}\right)^{-2} \leq  \max_{u \geq 1} \left[ \left(1 - \frac{2 \log u}{u} \right)^{-2} \right] \leq 16.
\end{align*}
Hence:
\begin{align*}
    \left(1 - \frac{2 \log n}{n}\right)^{\lfloor n/2\rfloor - 1} &\leq 16 \left(1 - \frac{2 \log n}{n}\right)^{n/2} \\
    &\leq 16 \exp \left( n/2 \log \left(1 - \frac{2 \log n}{n} \right) \right) \\
    &\leq 16 \exp \left(- n/2 \times \frac{2 \log n}{n} \right) \leq 16/n.
\end{align*}

Coming back to the telescoping sum:
\begin{align*}
  \sum_{k=\lfloor n/2 \rfloor}^{n}  \Ee W^{-1}_{k-1} 
  & \leq \frac{\lambda n}{2(1-\gamma)} \left[ \frac{8( \log n) \sigma^2}{\lambda^2 n} + \frac{16}{n}\frac{M_\cH \|r\|^2_{L^2(p)}}{\lambda^2 }  \right] + \frac{8 \sigma^2}{1-\gamma} \frac{\log n}{\lambda} .
\end{align*}
Dividing by $n - \lfloor n/2 \rfloor +1 \geq n/2$:
\begin{align*}
  \frac{1}{n - \lfloor n/2 \rfloor +1}\sum_{k=\lfloor n/2 \rfloor}^{n}  \Ee W^{-1}_{k-1} 
  & \leq \frac{1 }{(1-\gamma)} \left[ \frac{8 (\log n) \sigma^2}{\lambda n} + \frac{16}{n}\frac{M_\cH \|r\|^2_{L^2(p)}}{\lambda }  \right] + \frac{16 \sigma^2}{1-\gamma} \frac{\log n}{\lambda n} .
\end{align*}
All the terms are of order $\tilde O(\frac{\log n}{\lambda n})$.

Consider the $n$-th tail averaged iterate: $$V_n^{\textit{(t)}} := \frac{1}{n - \lfloor n/2 \rfloor +1}\sum_{k=\lfloor n/2 \rfloor}^{n}  V_{k-1}.$$
Using Jensen's inequality, we have a bound on its distance to $V^*_\lambda$:
\begin{align*}
  \Ee \| V_n^{\textit{(t)}} - V^*_\lambda \|_{L^2(p)}^2 & \leq  \frac{16}{n}\frac{M_\cH \|r\|^2_{L^2(p)}}{\lambda (1-\gamma)}  + \frac{24 \sigma^2}{1-\gamma} \frac{\log n}{\lambda n} .
\end{align*}

Now we need to choose $\lambda$ such that $\rho_k \leq \min \{1/(2\lambda), \bar \rho \}$, for all $k$. Let $\lambda = \lambda_0 n^{-\frac{1}{2+\theta}}$.

$\bullet~$ For the first half, $\rho = \frac{2 \log n}{\lambda n}$, and $\rho \leq 1/(2\lambda)$ if and only if $\log n/n \leq 4$, which is true for $n \geq 9$.

Now $\rho \leq \bar \rho$  is equivalent to $\frac{2\log n}{ \lambda n} = (\log n) n^{\frac{1}{2+\theta} -1}/\lambda_0 \leq \bar \rho$. Since $\theta >-1$, $\frac{1}{2+\theta} -1 < 0$ and $(\log n) n^{\frac{1}{2+\theta} -1} / \bar \rho  \rightarrow 0$. In particular it is bounded for all $n \geq 1$. Hence using again:
    $$\underline{\lambda}^{\textit{(e)}}_\theta = \max \{(\log n) n^{\frac{1}{2+\theta} - 1} / \bar \rho ~|~ n \geq 1 \},$$ then for $\lambda_0 \geq \underline{\lambda}^{\textit{(e)}}_\theta$, $\rho \leq \bar \rho$ is satisfied.

$\bullet~$ For the second half, $\rho_k$ is decreasing with $k$, hence a sufficient condition is that: $$\frac{1}{\lambda \lfloor n/2 \rfloor} = \rho_{\lfloor n/2 \rfloor} \leq \min\{1/(2\lambda), \bar \rho\}.$$
For $n \geq 4$, $\lfloor n/2 \rfloor \geq 2$ and $\rho_{\lfloor n/2 \rfloor} \leq 1/(2 \lambda)$. On the other hand, the second condition reads:
$$\frac{1}{\lambda \lfloor n/2 \rfloor} = \frac{n^{\frac{1}{2+\theta}}}{\lambda_0 \lfloor n/2 \rfloor} \leq \frac{4n^{\frac{1}{2+\theta}-1}}{\lambda_0 } \leq \bar \rho ,$$
for $n \geq 2$. Since $\theta >-1$, $\frac{1}{2+\theta} -1 < 0$ and $4 n^{\frac{1}{2+\theta} -1} / \bar \rho  \rightarrow 0$. In particular it is bounded for all $n \geq 1$. Hence using:
    $$\underline{\lambda}^{\textit{(t)}}_\theta := \max \{ \max \{4 n^{\frac{1}{2+\theta} - 1} / \bar \rho ~|~ n \geq 1 \}, \underline{\lambda}^{\textit{(e)}}_\theta \},$$ then for $\lambda_0 \geq \underline{\lambda}^{\textit{(t)}}_\theta$, $\rho_k \leq \bar \rho$ is satisfied for all $k$.

For this specific choice of $\lambda$, we have the final bound:
\begin{align*}
   \Ee \|V_n^{\textit{(t)}} - V^* \|_{L^2(p)}^2 &\leq 2 \Ee \|V_n^{\textit{(t)}} - V^*_\lambda \|^2_\cH + 2 \|V^*_\lambda - V^* \|^2_{L^2(p)} \\
    &\leq \frac{32}{n}\frac{M_\cH \|r\|^2_{L^2(p)}}{\lambda (1-\gamma)}  + \frac{48 \sigma^2}{1-\gamma} \frac{\log n}{\lambda n} + \frac{2  \| \Sigma^{-\theta/2} V^*\|_\cH^2 \lambda_0^{1+\theta}}{(1-\gamma)^2} n^{-\frac{1+\theta}{2+\theta}} \\
    & \leq \frac{32 M_\cH \|r\|^2_{L^2(p)}}{ \lambda_0(1-\gamma)} n^{-\frac{1+\theta}{2+\theta}} + \frac{48 \sigma^2}{\lambda_0(1-\gamma)} (\log n)n^{-\frac{1+\theta}{2+\theta}} \\
    &\qquad\qquad\qquad\qquad\qquad\quad + \frac{2  \| \Sigma^{-\theta/2} V^*\|_\cH^2 \lambda_0^{1+\theta}}{(1-\gamma)^2} n^{-\frac{1+\theta}{2+\theta}}.
\end{align*}
Finally, we define $\underline{\lambda}_\theta := \max \{\underline{\lambda}_\theta^{(0)}, \underline{\lambda}_\theta^{(e)}, \underline{\lambda}_\theta^{(t)} \}$ which is used in the theorem as lower bound on~$\lambda_0$.
\end{proof}

\subsection{Stochastic TD with Markovian  sampling} \label{sec:proofs_markov}

We begin by reproducing Lemma~9 from~\cite{bhandari2018finite}: 
\begin{lemma}[Control of couplings] \label{lemma_cpl}
 Consider two random variables $X$ and $Y$ such that:
 $$ X \rightarrow x_n \rightarrow x_{n+\tau} \rightarrow Y $$
 forms a Markov chain, for some fixed $n\geq 1$ with $\tau>0$. Assume the Markov chain mixes at uniform geometric rate. Let $X'$ and $Y'$ denote independent copies drawn from the marginal distributions of $X$ and $Y$, so that $$\mathbb{P}(X'=\cdot, Y'=\cdot) = \mathbb{P}(X=\cdot) \otimes \mathbb{P}(Y=\cdot). $$ Then for any bounded function $h$:
 $$| \Ee [h(X, Y)] - \Ee[h(X', Y')]| \leq 2 \|h\|_\infty m \mu^\tau. $$
\end{lemma}
Note that, here, $\otimes$ does not refer to the outer product in the RKHS~$\cH$ but to the independent product of probability distributions.

Then we can state a descent lemma, similar to Lemma~\ref{lemma3}: 
\begin{lemma} \label{lemma_proj} Assume that $\|V^*_\lambda\|_\cH \leq B$ and that the Markov chain mixes geometrically. Let:
\begin{align*}
    \left\{
    \begin{array}{l}
        G^2 := 4 M_\cH^2 B^2  + \lambda^2 B^2 + M_\cH R^2/2 \\
        L := 12 M_\cH B + 2 \sqrt{M_\cH} R \\
        C :=2M_\cH B+ \lambda B + \sqrt{M_\cH} R \\
        C':= 8 M_\cH B^2 + 4 \sqrt{M_\cH} B R.
    \end{array}
\right.
\end{align*}

Then for $n \geq 1$ and $\tau >1$:
\begin{align} \label{eqn:desc_proj}
    \Ee W_n^0 & \leq (1-2\rho_n \lambda) \Ee W_{n-1}^0 -2 \rho_n (1-\gamma) \Ee W_{n-1}^{-1}  + 2\rho_n \left( 2 C' m \mu^\tau + LC \sum_{k=n-\tau}^{n-1} \rho_k \right) + 4 G^2 \rho_n^2.
\end{align}
\end{lemma}

\begin{proof}[Proof of Lemma~\ref{lemma_proj}]
Because of correlations between samples, the proof of Lemma~\ref{lemma3} breaks here:
$$ \Ee [\langle V_{n-1} - V^*_\lambda, (A_n - \lambda I) V_{n-1} + b_n \rangle_\cH ] \neq \Ee [\langle V_{n-1} - V^*_\lambda, (A - \lambda I) V_{n-1} + b \rangle_\cH ] .$$
A similar thing occurs in the variance term, where we cannot apply Lemma~\ref{lemma2}. An easy fix is to assume that what is inside the variance remains bounded a.s. This is allowed by our projection step. We can now assume that a.s., $\forall n, \|V_n \|_\cH \leq B$. Hence a.s.:
$$ \| A_n V_{n-1} \|_\cH \leq \| A_n \|_{\rm op} B \leq 2 M_\cH B. $$
Indeed, for $f \in \cH$:
\begin{align*}
    \| A_n f\|_\cH & \leq \| \Phi(x_n) \otimes \Phi(x_n') f\|_{\cH} +   \| \Phi(x_n) \otimes \Phi(x_n) f\|_{\cH} \\
    & \leq \big( |\langle f, \Phi(x'_n) \rangle_\cH | +  |\langle f, \Phi(x_n) \rangle_\cH |\big) \| \Phi(x_n) \|_\cH \\
    & \leq 2 \|f \|_\cH \sqrt{M_\cH} \sqrt{M_\cH} = 2 M_\cH \|f \|_\cH.
\end{align*}
Also, since the reward function is uniformly bounded by~$R$:
\begin{align*}
    \| b_n\|^2_\cH & = \|r(x_n) \Phi(x_n) \|^2_\cH \leq R^2 M_\cH.
\end{align*}
Finally, since $\Pi_B$ is a contraction mapping in $\cH$ norm, this will not impact the proof.
 
\textbf{Decomposition of errors.} Let us now reproduce the beginning of the proof of Lemma~\ref{lemma3}. We have this decomposition a.s.:
\begin{align*}
    W_n^0 & = \| V_n - V^*_\lambda\|_\cH^2 \\
    &=  \| \Pi_B [V_{n-1} + \rho_n ((A_n - \lambda I)V_{n-1} + b_n)] - \Pi_B V^*_\lambda \|_\cH^2 \\
    & \leq \| V_{n-1} + \rho_n ((A_n - \lambda I)V_{n-1} + b_n) - V^*_\lambda \|^2_\cH \\
    &= \| V_{n-1} - V^*_\lambda\|_\cH^2 + 2 \rho_n \langle V_{n-1} - V^*_\lambda, (A_n - \lambda I)V_{n-1} + b_n \rangle_\cH \\&\qquad\qquad\qquad\qquad\qquad\qquad\qquad\qquad + \rho_n^2 \| (A_n - \lambda I) V_{n-1} + b_n \|_\cH^2 \\
    &\leq W_{n-1}^0 + 2 \rho_n \langle V_{n-1} - V^*_\lambda, (A_n - \lambda I)V_{n-1} + b_n \rangle_\cH \\&\qquad\qquad\qquad\qquad\qquad\qquad\qquad\qquad + 2 \rho_n^2 \| (A_n - \lambda I) V_{n-1}\|^2 + 2 \rho_n^2 \| b_n \|^2 \\
    &\leq W_{n-1}^0 + 2 \rho_n \langle V_{n-1} - V^*_\lambda, (A_n - \lambda I)V_{n-1} +b_n \rangle_\cH \\
    &\qquad\qquad\qquad\qquad\qquad\qquad\qquad\qquad + 4\rho_n^2 (4M_\cH^2 B^2 + \lambda^2 B^2) +2 \rho_n^2 R^2 M_\cH .
\end{align*}
 
 Taking the expectation with respect to $\cF_n = \sigma(z_1,...,z_n)$ (where $z_i = (x_i, x'_i)$), we get three terms:
\begin{align*}
    \Ee W_n^0 & \leq \Ee W_{n-1}^0 + 2 \rho_n \Ee \left[ \langle V_{n-1} - V^*_\lambda, (A_n - \lambda I)V_{n-1} + b_n \rangle_\cH \right] \\
    &\qquad\qquad\qquad+ \rho_n^2 \underbrace{(16 M_\cH^2 B^2  + 4\lambda^2 B^2 +2 M_\cH R^2 )}_{:=4 G^2}.
\end{align*}

We then deal with the central expectation.
\begin{align*}
    \Ee \left[ \langle V_{n-1} - V^*_\lambda, (A_n - \lambda I)V_{n-1} + b_n \rangle_\cH \right] &=\Ee \left[ \langle V_{n-1} - V^*_\lambda, (A - \lambda I)V_{n-1} + b \rangle_\cH \right] \\
    &\qquad+ \Ee \left[ \langle V_{n-1} - V^*_\lambda, (A_n -A)V_{n-1} + (b_n - b) \rangle_\cH \right].
\end{align*}

The first term has already been treated in Lemma~\ref{prop5}:
\begin{align*}
    \Ee \left[ \langle V_{n-1} - V^*_\lambda, (A - \lambda I)V_{n-1} + b \rangle_\cH \right] \leq - (1-\gamma) \Ee W_{n-1}^{-1} - \lambda \Ee W^0_{n-1}.
\end{align*}

To control the remaining expectation (the bias), we must use a coupling argument. We use the notation:
$$ \zeta(V_{n-1}, z_n) := \langle V_{n-1} - V^*_\lambda, (A_n - A) V_{n-1} + (b_n - b) \rangle_\cH. $$
Note that in general:
$$\Ee \zeta(V_{n-1}, z_n) = \Ee [ \Ee [ \zeta(V_{n-1}, z_n) | \cF_{n-1}] ] \neq 0,$$
where $\cF_{k} = \sigma(z_1,...,z_k) =\sigma(z_1, V_1,...,z_k, V_k) $. The dependence between the random variables is summarized in the following diagram.

\begin{center}
\begin{tikzpicture}
\matrix[matrix of math nodes,column sep=2em,row
sep=3em,cells={nodes={circle,draw,minimum width=3em,inner sep=0pt}},
column 1/.style={nodes={rectangle,draw=none}},
column 5/.style={nodes={rectangle,draw=none}}] (m) {
\text{Markov process}: &
 z_1 & z_2 & z_3 & \cdots & z_{n-1} & z_{n}\\
\text{TD iterates}: & 
 V_1 & V_2 & V_3 & \cdots & V_{n-1} & V_{n}\\
};
\foreach \X in {2,3,4,5,6, 7}
{\ifnum\X<7
\draw[-latex] (m-1-\X) -- (m-1-\the\numexpr\X+1) node[midway,above]{};
\draw[-latex] (m-2-\X) -- (m-2-\the\numexpr\X+1) node[midway,above]{};
\fi
\ifnum\X=5
\else
\draw[-latex] (m-1-\X) -- (m-2-\X) node[pos=0.6,left]{};
\fi}
\end{tikzpicture}
\end{center}

Using the mixing assumption, we can control the deviation between the expectations of a bounded function of two iterates separated by $\tau$ steps, in the coupled \textit{v.s.}~the decoupled case. In other words, if~$\tau$ is large, we can almost consider the iterates are independent. This is achieved using Lemma~\ref{lemma_cpl}.

\textbf{Bounding the bias.} Our goal here is to find an upper-bound of  $\Ee [\zeta(V_{n-1}, z_n)]$. Let $\tau \in \Nn$, $\tau > 1$. This can be done in two steps: \begin{enumerate}[label=(\arabic*)]
    \item Relate $\Ee [\zeta(V_{n-1}, z_n)] $ to $\Ee [\zeta(V_{n-1-\tau}, z_n)] $, because ~$\zeta$ is Lipschitz in the first variable, as a quadratic function over a bounded domain. This is true almost surely, hence in  expectation.
    \item Relate $\Ee [\zeta(V_{n-1-\tau}, z_n)]$ to $\Ee [\zeta(V_{n-1-\tau}', z_n')] = 0$, where $V_{n-1-\tau}'$ and $z'_n$ are independent copies of $V_{n-1-\tau}$ and $z_n$ that are decoupled.
\end{enumerate}

(1) First we prove that $\zeta$ is $L$-Lipschitz in the first variable on the $\cH$ ball of radius~$B$: for fixed $V, V' \in \cH$ with norm bounded by $B$, and $z_n$:
\begin{align*}
    | \zeta(V, z_n) - \zeta(V', z_n) | &= \Big| \langle (A_n - A) V + b_n - b , V - V^*_\lambda \rangle_\cH \\
    &\qquad\qquad\qquad-  \langle  (A_n - A) V' + b_n - b , V' - V^*_\lambda \rangle_\cH \Big|\\
    &= \Big| \langle (A_n - A) V + b_n - b , V - V' \rangle_\cH \\
     &\qquad\qquad\qquad+ \langle (A_n - A) (V - V') , V' - V^*_\lambda \rangle_\cH \Big|,
     \end{align*}
where we have used the equality:
$$ \langle a, b \rangle - \langle c, d \rangle = \langle a, b-d \rangle + \langle a - c, d \rangle.$$
    \begin{align*} 
    | \zeta(V, z_n) - \zeta(V', z_n) | & \leq  \|(A_n - A) V + b_n - b\|_\cH \cdot \| V - V'\|_\cH \\
    &\qquad\qquad\qquad+ \|(A_n - A)(V - V')\|_\cH \cdot \|V' - V^*_\lambda\|_\cH \\
    & \leq (4M_\cH B + 2 \sqrt{M_\cH} R) \|V- V'\|_\cH + 8  M_\cH B \|V-V'\|_\cH \\
    &= L \|V-V'\|_\cH, 
\end{align*}
for $L:= 4M_\cH B + 2 \sqrt{M_\cH} R+ 8  M_\cH B$.

Then almost surely, since all the $V_k$ are such that $\|V_k \|_\cH \leq B$:
\begin{align*}
    \zeta(V_{n-1}, z_n) &\leq  \zeta(V_{n-1-\tau}, z_n) + |\zeta(V_{n-1}, z_n) - \zeta(V_{n-1-\tau}, z_n) | \\
    &\leq \zeta(V_{n-1-\tau}, z_n) + L \|V_{n-1} - V_{n-1-\tau} \|_\cH \\
    &\leq \zeta(V_{n-1-\tau}, z_n) + L \sum_{k=n-\tau}^{n-1} \| V_k - V_{k-1} \|_\cH \\
    &= \zeta(V_{n-1-\tau}, z_n) + L \sum_{k=n-\tau}^{n-1} \rho_k \| A_k V_{k-1} - \lambda V_{k-1} + b_k \|_\cH \\
    &\leq \zeta(V_{n-1-\tau}, z_n) + L \sum_{k=n-\tau}^{n-1} \rho_k \underbrace{(2M_\cH B+ \lambda B + \sqrt{M_\cH} R )}_{=:C}.
\end{align*}

Taking the expectation w.r.t. $\mathbb{P}(z_1,...,z_n)$:
$$ \Ee \zeta(V_{n-1}, z_n) \leq \Ee \zeta(V_{n-1-\tau}, z_n)  + LC \sum_{k=n-\tau}^{n-1} \rho_k . $$

(2) Then we use a coupling argument with Lemma~\ref{lemma_cpl}. First, we need to bound $\| \zeta \|_\infty$.

For fixed $V$, $z_n$, with $\|V\|_\cH \leq B$,  almost surely:
\begin{align*}
    | \zeta(V, z_n) | &= | \langle (A_n - A )V + b_n - b, V -V^*_\lambda \rangle_\cH | \\
    & \leq \|V-V^*_\lambda \|_\cH \Big(\|(A_n - A )V\|_\cH + \| b_n - b \|_\cH \Big) \\
    &\leq 2 B (4 M_\cH B + 2 \sqrt{M_\cH} R) =: C'.
\end{align*}

In Lemma~\ref{lemma_cpl}, set $X = (z_1,...,z_{n-1-\tau})$ and  $Y = z_n$. Since:
 $$ X \rightarrow x_{n-\tau} \rightarrow x_n \rightarrow Y $$
 forms a Markov chain, then let $X'$ and $Y'$ denote independent copies drawn from the marginal distributions of $X$ and $Y$, so that $\mathbb{P}(X'=\cdot, Y'=\cdot) = \mathbb{P}(X=\cdot) \otimes \mathbb{P}(Y=\cdot)$.  Then applying Lemma~\ref{lemma_cpl} to the function $h:(X, Y) \rightarrow \zeta(V_{n-1-\tau}, z_n)$ (recalling that  $V_{n-1-\tau}$ is fully determined by the values of $X$):
 $$| \Ee [h(X, Y)] - \Ee[h(X', Y')]| \leq 2 \|h\|_\infty m \mu^\tau. $$
 In other words:
  $$ | \Ee \zeta(V_{n-1-\tau}, z_n) - \Ee \zeta(V'_{n-1-\tau}, z_n')| \leq 2C' m \mu^\tau. $$
  By definition of the random variables $X', Y'$:
  $$\Ee \zeta(V'_{n-1-\tau}, z_n') = \Ee [ \Ee [ \zeta(V'_{n-1-\tau}, z_n') | V'_{n-1-\tau}]] = 0.$$
  
  Putting everything together, we get:
\begin{align*}
    \Ee \zeta(V_{n-1}, z_n) &\leq \Ee \zeta(V_{n-1-\tau}, z_n)  + LC \sum_{k=n-\tau}^{n-1} \rho_k  \\
    & \leq 2C' m\mu^\tau + LC \sum_{k=n-\tau}^{n-1} \rho_k.
\end{align*} 

Using this upper-bound is interesting if $m \mu^\tau$ is of the order of $\sum_{k=n-\tau}^{n-1} \rho_k$. Else (for small $n$), one can always choose $\tau = n-1$, so that, 
because $V_0$ is deterministic:
\begin{align*}
    \Ee \zeta(V_{n-1}, z_n) &\leq \underbrace{\Ee \zeta(V_0, z_n)}_{=0}  + LC \sum_{k=1}^{n-1} \rho_k .
\end{align*} 
\end{proof}

\begin{proof}[Proof of Theorem~\ref{thm2}]
We use a constant step size~$\rho$. From Lemma~\ref{lemma_proj}:
\begin{align*}
    \Ee W_n^0 & \leq (1-2\rho \lambda) \Ee W_{n-1}^0 - 2\rho (1-\gamma) \Ee W_{n-1}^{-1}  + 2\rho \left( 2 C' m \mu^\tau + LC \tau \rho \right) + 4 G^2 \rho^2.
\end{align*}

In particular, we choose~$\tau$ such that $ \mu^\tau = \rho$, that is $\tau = \frac{\log \rho}{\log \mu}= \frac{\log (1/\rho)}{\log (1/\mu)}$. Then:
\begin{align*}
    \Ee W_n^0 & \leq (1-2\rho \lambda)\Ee W_{n-1}^0 - 2\rho (1-\gamma)\Ee W_{n-1}^{-1}  + 2\rho \left( 2 C' m \rho + LC  \rho \frac{\log(1/\rho)}{\log(1/\mu)} \right) + 4 G^2 \rho^2 \\
    & \leq (1-2\rho \lambda) \Ee W_{n-1}^0 - 2\rho (1-\gamma)\Ee W_{n-1}^{-1}  + \rho^2 \left( \underbrace{4 C' m  + 2 LC   \frac{\log(1/\rho)}{\log(1/\mu)} + 4 G^2 }_{=:4 \tilde \sigma_{\lambda, \rho}^2} \right).
\end{align*}

This expression is similar to~(\ref{eqn:rec}). Adapting the proof of Thm.~\ref{thm1} (b), we obtain:
\begin{align*}
    \Ee \|V^{\textit{(e)}}_n - V^*_\lambda \|_{L^2(p)}^2 &\leq \frac{ (1 - 2\rho \lambda)^n}{ 1 - (1 - 2\rho \lambda)^n} \frac{M_\cH \|r\|_{L^2(p)}^2}{\lambda(1-\gamma)} + \frac{2  \rho \sigma^2_{\lambda, \rho}}{1-\gamma},
\end{align*}
with $V^{\textit{(e)}}_n = \frac{\sum_{k=1}^n  (1-2\rho \lambda)^{n-k}  V_{k-1}}{\sum_{k=1}^n (1-2\rho \lambda)^{n-k}}$ the exponentially weighted average iterate.

Finally:
\begin{align*}
    \Ee \|V^{\textit{(e)}}_n - V^* \|_{L^2(p)}^2 &\leq \frac{ 2(1 - 2\rho \lambda)^n}{ 1 - (1 - 2\rho \lambda)^n} \frac{M_\cH \|r\|_{L^2(p)}^2}{\lambda(1-\gamma)} + \frac{4  \rho \sigma^2_{\lambda, \rho}}{1-\gamma} + \frac{2  \| \Sigma^{-\theta/2} V^*\|_\cH^2}{(1-\gamma)^2} \lambda^{1+\theta}.
\end{align*}

Note that $\sigma^2_{\lambda, \rho}$ depends on $\lambda$, $\rho$, and $B$. We look at two cases: \begin{enumerate}[label=(\roman*)]
    \item we are given an oracle on $B$ that does not depend on~$\lambda$.
    \item we use the bound of order $O(1/\lambda)$ given by Prop.~\ref{prop3}:
    $$B = \frac{\sqrt{M_\cH} \|r\|_{L^2(p)}}{\lambda }. $$
\end{enumerate}

\textbf{Case \textit{(i)}: with oracle.} For a fixed $\lambda$ (later chosen to be the optimal one), assume we know a bound~$B$ on $\|V^*_\lambda\|_\cH$. Then $B = O(1)$,  and assuming $\lambda = O(1)$, we only keep track of the dependence in~$\mu$ and put all the other constants in $O(1)$:
$$\sigma^2_{\lambda, \rho} = O\left( \frac{\log(1/\rho)}{\log(1/\mu)} \right) + O(1).  $$
Let us look for $\lambda$ of the form $\lambda = n^{-\alpha}$ with $\alpha \in (0, 1)$:
\begin{align*}
    \Ee \| V^{\textit{(e)}}_n - V^* \|_{L^2(p)}^2 &\leq O\left( \frac{ (1 - 2\rho \lambda)^n}{ 1 - (1 - 2\rho \lambda)^n} \frac{1}{\lambda} \right) + O\left( \rho \frac{\log(1/\rho)}{\log (1/\mu)}  \right) + O(\rho)+ O\left( \lambda^{1+\theta}\right).
\end{align*}
Let us now set $\rho = \frac{\log n}{2\lambda n}$:
\begin{align*}
    \Ee \| V^{\textit{(e)}}_n - V^* \|_{L^2(p)}^2 &\leq O\left( \frac{1}{n\lambda} \right)+ O\left( \frac{\log n}{\lambda n} \frac{\log(1/\rho)}{\log (1/\mu)}  \right) + O(\rho) +O\left( \lambda^{1+\theta}\right).
\end{align*}
Expressing everything with $n$ only:
\begin{align*}
    \Ee \| V^{\textit{(e)}}_n - V^* \|_{L^2(p)}^2 &\leq O\left(  n^{\alpha-1}\right)+ O\left( \frac{(\log n)^2 n^{\alpha-1} }{\log (1/\mu)}  \right) +O\left((\log n)  n^{\alpha-1}\right) + O\left( n^{-\alpha(1+\theta)}\right).
\end{align*}
The first and third terms are smaller than the second one. We can choose $\alpha$ such that: $\alpha - 1 = -\alpha(1+\theta) \iff \alpha = \frac{1}{2 + \theta}$, hence we get the convergence rate:
$$\Ee \left[ \| V^{\textit{(e)}}_n - V^* \|_{L^2(p)}^2 \right] \leq O \left( \frac{(\log n)^2 n^{-\frac{1+\theta}{2+\theta}} }{\log(1/\mu)} \right).$$

\textbf{Case \textit{(ii)}: without oracle.}

Now $B=O(1/\lambda)$. Let us unroll all the constants to see the full dependencies:
\begin{align*}
    \sigma^2_{\lambda,\rho} &= C'm + \frac{1}{2} LC \frac{\log(1/\rho)}{\log (1/\mu)} + G^2 \\
    &= 8 m M_\cH B^2 + 4m \sqrt{M_\cH}R B  \\
    &\qquad\qquad\qquad\qquad\qquad+ \left(12 M_\cH B + 2 \sqrt{M_\cH} R \right) \left(2 M_\cH B + \lambda B+ \sqrt{M_\cH} R \right)\frac{\log(1/\rho)}{2\log (1/\mu)} \\&\qquad\qquad\qquad\qquad\qquad + 4 M_\cH^2 B^2 + \lambda^2 B^2 + M_\cH R^2/2 \\
    &= B^2 \left(8 m M_\cH + 4 M_\cH^2 +  \lambda^2 + 12 M_\cH^2 \frac{\log(1/\rho)}{\log (1/\mu)} + 6 \lambda M_\cH \frac{\log(1/\rho)}{\log (1/\mu)} \right) \\
    &\qquad\qquad + B \left( 4 m \sqrt{M_\cH} R + 8 M_\cH^{3/2} R \frac{\log(1/\rho)}{\log (1/\mu)} + \lambda \sqrt{M_\cH} R \frac{\log(1/\rho)}{\log (1/\mu)} \right) \\
    &\qquad\qquad\qquad\qquad + \left( M_\cH R^2/2 + M_\cH R^2 \frac{\log(1/\rho)}{\log (1/\mu)} \right).
\end{align*}

We focus on the case $\lambda = O(1)$, so this simplifies a bit to:
\begin{align*}
    \sigma^2_{\lambda, \rho} = O(B^2) + O\left( \frac{\log(1/\rho)}{\log (1/\mu)} B^2 \right) + O(B) + O\left( \frac{\log(1/\rho)}{\log (1/\mu)} B \right) + O(1) + O\left(\frac{\log(1/\rho)}{\log (1/\mu)}\right).
\end{align*}
On the other hand, $B = O(1/\lambda)$, hence:
\begin{align*}
    \sigma^2_{\lambda, \rho} = O(1/\lambda^2) + O\left( \frac{\log(1/\rho)}{\lambda^2\log (1/\mu)} \right) + O\left(\frac{1}{\lambda}\right) + O\left( \frac{\log(1/\rho)}{\lambda \log (1/\mu)} \right) + O(1) + O\left(\frac{\log(1/\rho)}{\log (1/\mu)}\right).
\end{align*}

Let us look for $\lambda$ of the form $\lambda = n^{-\alpha}$ with $\alpha \in (0, 1)$. 

In this case $\sigma^2_{\lambda, \rho} = O(1/\lambda^2) + O\left( \frac{\log(1/\rho)}{\log (1/\mu)} 1/\lambda^2 \right) $ and:
\begin{align*}
    \Ee \|  V^{\textit{(e)}}_n - V^* \|_{L^2(p)}^2 &\leq O\left( \frac{ (1 - 2\rho \lambda)^n}{ 1 - (1 - 2\rho \lambda)^n} \frac{1}{\lambda} \right) + O\left(\frac{\rho}{\lambda^2}\right) + O\left( \frac{\rho}{\lambda^2} \frac{\log(1/\rho)}{\log (1/\mu)}  \right) + O\left( \lambda^{1+\theta}\right).
\end{align*}
Let us now set $\rho = \frac{\log n}{2\lambda n}$:
\begin{align*}
    \Ee \|  V^{\textit{(e)}}_n - V^* \|_{L^2(p)}^2 &\leq O\left( \frac{1}{n\lambda} \right)+O\left( \frac{\log n}{\lambda^3 n}  \right)  + O\left( \frac{\log n}{\lambda^3 n} \frac{\log(1/\rho)}{\log (1/\mu)}  \right) + O\left( \lambda^{1+\theta}\right).
\end{align*}
Expressing everything with $n$ only:
\begin{align*}
    \Ee \|  V^{\textit{(e)}}_n - V^* \|_{L^2(p)}^2 &\leq O\left( n^{\alpha-1}\right)+O\left((\log n) n^{3\alpha-1} \right)  +O\left( \frac{(\log n)^2 n^{3\alpha-1} }{\log (1/\mu)}  \right) + O\left( n^{-\alpha(1+\theta)}\right).
\end{align*}
The first and second term are smaller than the third one. We can choose $\alpha$ such that: $3\alpha - 1 = -\alpha(1+\theta) \iff \alpha = \frac{1}{4 + \theta}$, hence we get the convergence rate:
$$\Ee \left[ \|  V^{\textit{(e)}}_n - V^* \|_{L^2(p)}^2 \right] \leq O \left( \frac{(\log n)^2 n^{-\frac{1+\theta}{4+\theta}}}{\log(1/\mu)} \right).$$
\end{proof}

\begin{corollary} \label{cor1}
Assuming~\textup{\ref{hypo:??}} and that the samples are produced by a Markov chain with uniform geometric mixing~\hyperref[hypo:???]{\textbf{\textup{(A3)}}}, the projected $\tau$-Skip-TD iterations~(\ref{eqn:skip_TD}) are such that:

(i) Using $\lambda = n^{-\frac{1}{2+\theta}}$, a constant step size $\rho = \frac{\log n}{2\lambda n}$, $\tau = \lceil \frac{\log(1/\rho)}{\log(1/\mu)}+1 \rceil$, and a projection radius~$B$ which is provided by an oracle and such that $\|V^*_\lambda\|_\cH \leq B$, then:
\begin{align}
    \Ee \left[ \| V^{\textit{(e)}}_{n/\tau} - V^* \|_{L^2(p)}^2 \right] \leq O \left( \frac{(\log n) n^{-\frac{1+\theta}{2+\theta}} }{\log(1/\mu)} \right).
\end{align}
(ii) Using $\lambda = n^{-\frac{1}{4+\theta}}$, $\rho = \frac{\log n}{2\lambda n}$, $\tau = \lceil \frac{\log(1/\rho)}{\log(1/\mu)}+1 \rceil$, and the projection radius~$B$ of Prop.~\ref{prop3}, then:
\begin{align}
    \Ee \left[ \| V^{\textit{(e)}}_{n/\tau} - V^* \|_{L^2(p)}^2 \right] \leq O \left( (\log n) n^{-\frac{1+\theta}{4+\theta}} \right),
\end{align}
assuming that $n$ is a multiple of~$\tau$, with $V^{\textit{(e)}}_n = \sum_{k=1}^n  (1-2\rho \lambda)^{n-k}  V_{k-1} /\sum_{j=1}^n (1-2\rho \lambda)^{n-j}$.
\end{corollary}

\begin{proof}[Proof of Corollary~\ref{cor1}]
We consider the iterates~(\ref{eqn:skip_TD}), for some positive integer $\tau$ to be chosen later. The beginning of the proof of Lemma~\ref{lemma_proj} can be reproduced:
\begin{align*}
    \Ee W_n^0 & \leq \Ee W_{n-1}^0 + 2 \rho_n \Ee \left[ \langle V_{n-1} - V^*_\lambda, (A_{n\tau} - \lambda I)V_{n-1} + b_{n\tau} \rangle_\cH \right] + 4 \rho_n^2 G^2 \\
    & \leq (1-2\rho_n \lambda) \Ee W_{n-1}^0 -2 \rho_n (1-\gamma) \Ee W_{n-1}^{-1}   + 4 G^2 \rho_n^2+ 2\rho_n \Ee \zeta(V_{n-1}, z_{n\tau}).
\end{align*}
The only difference is that we now consider
$ \Ee \zeta(V_{n-1}, z_{n \tau}) $ instead of $\Ee \zeta(V_{n-1}, z_n)$. To bound it, we do not need the step (1) (which exploits the fact that~$\zeta$ is Lipschitz), and directly go to step (2). The dependencies between the random variables are now:
\begin{center}
\begin{tikzpicture}
\matrix[matrix of math nodes,column sep=1.8em,row
sep=3em,cells={nodes={circle,draw,minimum width=4.4em,inner sep=0pt}},
column 1/.style={nodes={rectangle,draw=none}},
column 4/.style={nodes={rectangle,draw=none}}] (m) {
\text{Markov process}: &
 z_{\tau(n-1)} & z_{\tau(n-1)+1}  & \cdots & z_{\tau n -1} & z_{\tau n}\\
\text{Skip-TD iterates}: & 
 V_{n-1} & & & & V_{n}\\
};
\foreach \X in {2,3,4,5,6}
{\ifnum\X<6
\draw[-latex] (m-1-\X) -- (m-1-\the\numexpr\X+1) node[midway,above]{};
\fi
\ifnum\X=5
\else
\fi}

\draw[-latex] (m-2-2) -- (m-2-6) node[midway,above]{};
\draw[-latex] (m-1-2) -- (m-2-2) node[pos=0.6,left]{};
\draw[-latex] (m-1-6) -- (m-2-6) node[pos=0.6,left]{};
\end{tikzpicture}
\end{center}

Applying again Lemma~\ref{lemma_cpl}, we get the upper-bound:
 $$ | \Ee \zeta(V_{n-1}, z_{n\tau}) - \Ee \zeta(V'_{n-1}, z_{n\tau}')| \leq 2C' m \mu^{\tau-1},$$
 where $V'_{n-1}$, and $z_{n\tau}'$ are independent copies such that $\Ee \zeta(V'_{n-1}, z_{n\tau}')=0$.
 
 Now, using a constant step size~$\rho$, we set $\tau := \lceil \frac{\log(1/\rho)}{\log(1/\mu)}+1 \rceil$, such that $ \mu^{\tau - 1} \leq \rho$. Then:
 \begin{align*}
    \Ee W_n^0 &\leq (1-2\rho \lambda) \Ee W_{n-1}^0 -2 \rho (1-\gamma) \Ee W_{n-1}^{-1}   + 4 G^2 \rho^2+ 4\rho^2 C'm.
\end{align*}
Now we can do the same proof as for Theorem~\ref{thm2} with $\sigma^2_{\lambda, \rho} = C'm + G^2$,  now independent of~$\rho$:
\begin{align*}
    \Ee \|V^{\textit{(e)}}_n - V^* \|_{L^2(p)}^2 &\leq \frac{ 2(1 - 2\rho \lambda)^n}{ 1 - (1 - 2\rho \lambda)^n} \frac{M_\cH \|r\|_{L^2(p)}^2}{\lambda(1-\gamma)} + \frac{4  \rho \sigma^2_{\lambda, \rho}}{1-\gamma} + \frac{2  \| \Sigma^{-\theta/2} V^*\|_\cH^2}{(1-\gamma)^2} \lambda^{1+\theta}.
\end{align*}

\textbf{Case \textit{(i)}: with oracle.} Now $\sigma^2_{\lambda, \rho} =  O(1) $.
We look for $\lambda$ of the form $\lambda = n^{-\alpha}$, $\alpha \in (0, 1)$:
\begin{align*}
    \Ee \| V^{\textit{(e)}}_n - V^* \|_{L^2(p)}^2 &\leq O\left( \frac{ (1 - 2\rho \lambda)^n}{ 1 - (1 - 2\rho \lambda)^n} \frac{1}{\lambda} \right) +  O(\rho)+ O\left( \lambda^{1+\theta}\right).
\end{align*}
Let us now set $\rho = \frac{\log n}{2\lambda n}$:
\begin{align*}
    \Ee \| V^{\textit{(e)}}_n - V^* \|_{L^2(p)}^2 &\leq O\left( \frac{1}{n\lambda} \right) + O(\rho) +O\left( \lambda^{1+\theta}\right).
\end{align*}
Of course, to compute the $n$-th iteration, one needs to generate $\tau n$ samples from the Markov chain. So for a fair comparison, we must look at the convergence of~$V_{n/\tau}$ (assuming $n$ is a multiple of~$\tau$ for simplicity):
\begin{align*}
    \Ee \| V^{\textit{(e)}}_{n/\tau} - V^* \|_{L^2(p)}^2 &\leq O\left( \frac{\tau}{n\lambda} \right) + O(\rho) +O\left( \lambda^{1+\theta}\right).
\end{align*}
$\tau$ is such that: 
\begin{align*}
    \tau &= O\left(\frac{\log(1/\rho)}{\log(1/\mu)} \right) = O\left(\frac{\log n}{\log(1/\mu)} \right).
\end{align*}
Expressing everything with $n$ only:
\begin{align*}
    \Ee \| V^{\textit{(e)}}_{n/\tau} - V^* \|_{L^2(p)}^2 &\leq O\left( \frac{\log n}{\log(1/\mu)} n^{\alpha-1}\right) +O\left((\log n)  n^{\alpha-1}\right) + O\left( n^{-\alpha(1+\theta)}\right).
\end{align*}
Choosing $\alpha$ such that: $\alpha - 1 = -\alpha(1+\theta) \iff \alpha = \frac{1}{2 + \theta}$, we get the convergence rate:
$$\Ee \left[ \| V^{\textit{(e)}}_{n/\tau} - V^* \|_{L^2(p)}^2 \right] \leq O \left( \frac{(\log n) n^{-\frac{1+\theta}{2+\theta}}}{\log(1/\mu)} \right).$$

\textbf{Case \textit{(ii)}: without oracle.} Using $B = O(1/\lambda)$, now:
\begin{align*}
    \sigma^2_{\lambda, \rho} = O(1/\lambda^2)  + O\left(\frac{1}{\lambda}\right)  + O(1) .
\end{align*}
Let us look for $\lambda$ of the form $\lambda = n^{-\alpha}$ with $\alpha \in (0, 1)$. We also set  $\rho = \frac{\log n}{2\lambda n}$.  In this case $\sigma^2_{\lambda, \rho} = O(1/\lambda^2)  $ and:
\begin{align*}
    \Ee \|  V^{\textit{(e)}}_n - V^* \|_{L^2(p)}^2 &\leq O\left( \frac{ (1 - 2\rho \lambda)^n}{ 1 - (1 - 2\rho \lambda)^n} \frac{1}{\lambda} \right) + O\left(\frac{\rho}{\lambda^2}\right) + O\left( \lambda^{1+\theta}\right) \\
    & \leq O\left(  \frac{1}{n\lambda} \right) + O\left(\frac{\rho}{\lambda^2}\right) + O\left( \lambda^{1+\theta}\right).
\end{align*}
If~$n$ is a multiple of~$\tau$:
\begin{align*}
    \Ee \| V^{\textit{(e)}}_{n/\tau} - V^* \|_{L^2(p)}^2 &\leq O\left( \frac{\tau}{n\lambda} \right)  +O\left(\frac{\rho}{\lambda^2}\right) +O\left( \lambda^{1+\theta}\right).
\end{align*}
$\tau$ is such that: 
\begin{align*}
    \tau &= O\left(\frac{\log(1/\rho)}{\log(1/\mu)} \right) = O\left(\frac{\log n}{\log(1/\mu)} \right).
\end{align*}
Expressing everything with $n$ only:
\begin{align*}
    \Ee \| V^{\textit{(e)}}_{n/\tau} - V^* \|_{L^2(p)}^2 &\leq O\left( \frac{\log n}{\log(1/\mu)} n^{\alpha-1}\right) +O\left((\log n)  n^{3\alpha-1}\right) + O\left( n^{-\alpha(1+\theta)}\right).
\end{align*}
Choosing $\alpha$ such that: $3\alpha - 1 = -\alpha(1+\theta) \iff \alpha = \frac{1}{4 + \theta}$, we get the convergence rate:
$$\Ee \left[ \| V^{\textit{(e)}}_{n/\tau} - V^* \|_{L^2(p)}^2 \right] \leq O \left( (\log n) n^{-\frac{1+\theta}{4+\theta}} \right).$$
\end{proof}

\section{Experimental design} \label{sec:implem}

\subsection{Geometric mixing of the Markov chain} \label{subsec:mixing}

\begin{lemma} Consider the Markov chain defined on the torus $[0, 1]$ by:
\begin{itemize}
    \item with probability $\e$, $x_{n+1} \sim \cU([0, 1])$;
    \item with probability $1-\e$, $x_{n+1} = x_n$.
\end{itemize}
This Markov chain mixes to the uniform distribution at uniform geometric rate $(1-\e)$:
\begin{align*}
    \sup_{x \in [0, 1]} d_{TV} \big( \Pp( x_n \in \cdot | x_0=x), \cU([0, 1]) \big) \leq (1-\e)^n.
\end{align*}
\end{lemma}

\begin{proof}
Let $x \in [0, 1]$, $p = \cU([0, 1])$ the uniform distribution, and $p_n := \Pp( x_n \in \cdot | x_0=x)$.

We will show that:
$$d_{TV} \left( p_n, p \right) \leq (1-\e)^n.$$

For $n=1$, we have:
\begin{align*}
    p_1 = \Pp( x_1 \in \cdot | x_0=x) = \e p + (1 - \e) \delta_{x}.
\end{align*}
Then for $n=2$:
\begin{align*}
    \Pp( x_2 \in \cdot | x_0=x, x_1) = \e p + (1 - \e) \delta_{x_1}.
\end{align*}
Taking the marginal with respect to $x_1 | x_0$:
\begin{align*}
   p_2 &= \Pp( x_2 \in \cdot | x_0=x) \\
   &= \int (\e p + (1 - \e) \delta_{x_1}) dp_1(x_1) \\
   &= \e p + (1 - \e) \int  \delta_{x_1} (\e p(x_1) + (1-\e) \delta_x(x_1)) dx_1 \\
   &= \e p + \e (1-\e) p + (1-\e)^2 \delta_x.
\end{align*}

A simple recursion on $n$ shows that, for $n \geq 1$:
\begin{align*}
    p_n &= (\e + (1-\e) \e + ... + (1-\e)^{n-1} \e) p + (1-\e)^n \delta_x  \\
    & = (1 - (1-\e)^n) p + (1-\e)^n \delta_x.
\end{align*}
Hence:
\begin{align*}
    d_{TV} \left( p_n, p \right) &= \sup_{A \in \cA} \Big| p_n(A) - p(A) \Big| \\
    & = (1-\e)^n \sup_{A \in \cA} \Big| \delta_x(A) - p(A) \Big| \\
    &\leq (1-\e)^n.
\end{align*}
\end{proof}

\subsection{Implementation details} \label{subsec:implem}

The ``kernel trick'' enables an implementation of the non-parametric TD algorithm up to iteration $n$, which only uses the kernel matrix with entries $K_{i,j} := K(x_i, x_j)$, for $1 \leq i, j \leq n+1$.

Each value function $V_k$, for $1 \leq k \leq n$ belongs to the span of the basis of functions $(\Phi(x_j))_{1\leq j\leq k}$:
$$V_k = \sum_{j=1}^k \alpha_{k, j} \Phi(x_j). $$
Hence $V_k$ is represented in memory by the vector $(\alpha_{k, j})_{1 \leq j \leq k}$.

The TD iterations are equivalent to filling the lower-triangular matrix $\alpha$:
\begin{align*}\left\{
    \begin{array}{rrll}
    \alpha_{1, 1} &=& \rho_1 r(x_1) &\\
    \alpha_{k, j} &=& (1 - \rho_k \lambda) \alpha_{k-1, j} & \text{for }  1\leq j < k \leq n\\
        \alpha_{k, k} &=& \rho_k r(x_k) + \rho_k \sum_{j=1}^{k-1} \alpha_{k-1, j} \left( \gamma  K_{j, k+1} - K_{j, k} \right) & \text{for }  1\leq k \leq n.   \end{array}
\right.
\end{align*}
At inference time, for $x \in \cX$, $V_k(x)$ can be computed from $\alpha$ and the vector $(K(x_j, x))_{1\leq j \leq k}$:
$$V_k(x) = \sum_{j=1}^k \alpha_{k, j} K(x_j, x). $$
Finally, averaging can be performed by simple operations on $\alpha$, which correspond to exchanging the indices of a triangular sum. Indeed, if:
$$V_n^{\textit{(e)}} = \sum_{k=1}^n  w_{k, n}  V_{k-1},$$
for instance with $w_{k, n} := (1-\rho \lambda)^{n-k}/ \sum_{k=1}^n (1-\rho \lambda)^{n-k}$, then, using that $V_0=0$:
\begin{align*}
    V_n^{(e)} &=  \sum_{k=2}^n  w_{k, n} \sum_{j=1}^{k-1} \alpha_{k-1, j} \Phi(x_j) \\
    & = \sum_{1 \leq j < k \leq n} w_{k, n} \alpha_{k-1, j} \Phi(x_j) \\
    & = \sum_{j=1}^{n-1} \Phi(x_j) \sum_{k=j+1}^n w_{k, n} \alpha_{k-1, j} \\
    &= \sum_{j=1}^{n-1} \alpha_{n, j}^{(e)} \Phi(x_j),
\end{align*}
with $\alpha_{n, j}^{(e)} :=\sum_{k=j+1}^n w_{k, n} \alpha_{k-1, j}$.

This implementation requires the storage of $O(n^2)$ values and $O(n^2)$ computations  to compute~$V_n$. In our Python implementation, the limiting factor is the computation time of the kernel matrix. When $n \geq 1500$ and $K_2$ is used (empirically, the eigenvalues of the kernel matrix have a fast decrease), we use an incomplete Cholesky decomposition~\cite{bach2002kernel} with maximal rank 100 to approximate the kernel matrix. It is computed online with a fast Cython implementation, and does not require the compute the whole kernel matrix. Overall, the CPU time for computing $V_n$ for $n=2000$ is approximately 20 seconds on a standard laptop. Running all the experiments of this paper took a few hours.

\subsection{Additional experiments} \label{subsec:exps}

We test the robustness of TD to inexact estimations of~$\theta$, hence resulting in too large or too small~$\lambda$. If~$\theta$ is under-estimated, our theorems still guarantee convergence for $\theta > -1$, but not if it is over-estimated. In Fig.~\ref{fig2}, we plot the convergence of the averaged iterates for different values of $\theta$, smaller or larger than the optimal~$\theta=-1/4$ (standard deviations have been removed for readability). Fig.~\ref{fig2} shows that the convergence is quite robust and gives similar results for~$\theta=0$ or $\theta=-1/2$. A strongly overestimated~$\theta=1$ shows a slow convergence (not covered by our theorems). However, as expected, with $\theta=-1$, the algorithm does not converge.
\begin{figure}[ht]
    \centering
    \includegraphics[width=0.47\linewidth,trim={0 0 0.2cm 0},clip]{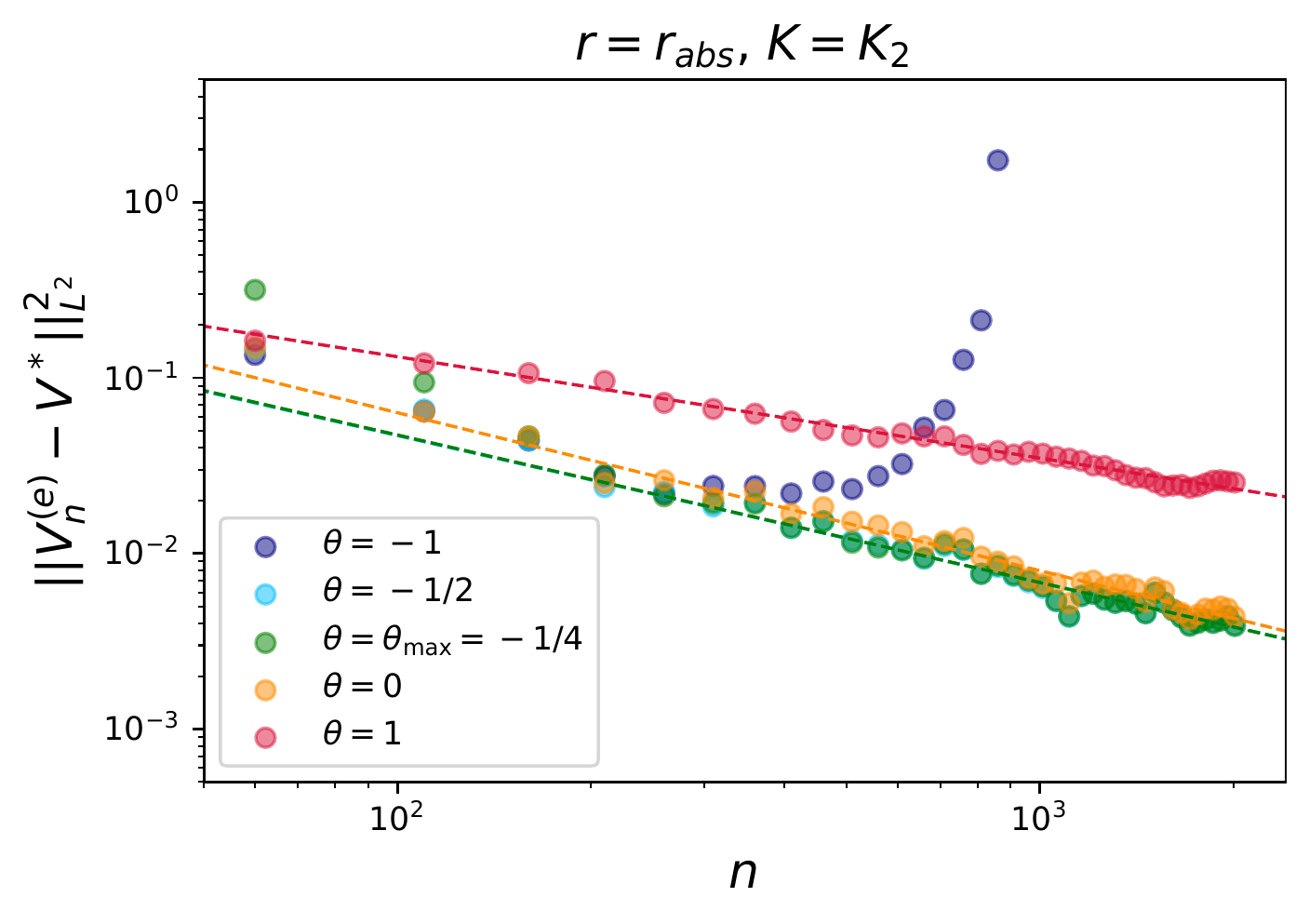}
    \caption{Convergence of the averaged TD iterates as in Thm.~\ref{thm1}\ref{hypo:thm1_b}} with over and underestimated values of~$\theta$.
    \label{fig2}
\end{figure}

Finally, we compare TD and $\tau$-Skip-TD, with $\tau$ prescribed by Cor.~\ref{cor1}. Computing this~$\tau$ requires the access to an oracle on the mixing parameter~$\mu$ ($\mu=1-\e$ in our example). We then use $\tau = \lceil \frac{\log(1/\rho)}{\log(1/\mu)}+1 \rceil$. We compare the results of TD and $\tau$-Skip-TD for two different values of~$\e$. We expect similar convergence rates, but with different constants. The results are plotted in Figure~\ref{fig3}. For the fast mixing chain ($\e=0.8$), we get comparable results. For the slowly mixing chain ($\e=0.2$), plain TD seems faster, although maybe the asymptotic regime has not been reached yet for $n=2000$.

\begin{figure}[ht]
    \centering
    \includegraphics[width=0.47\linewidth,trim={0 0 0.2cm 0},clip]{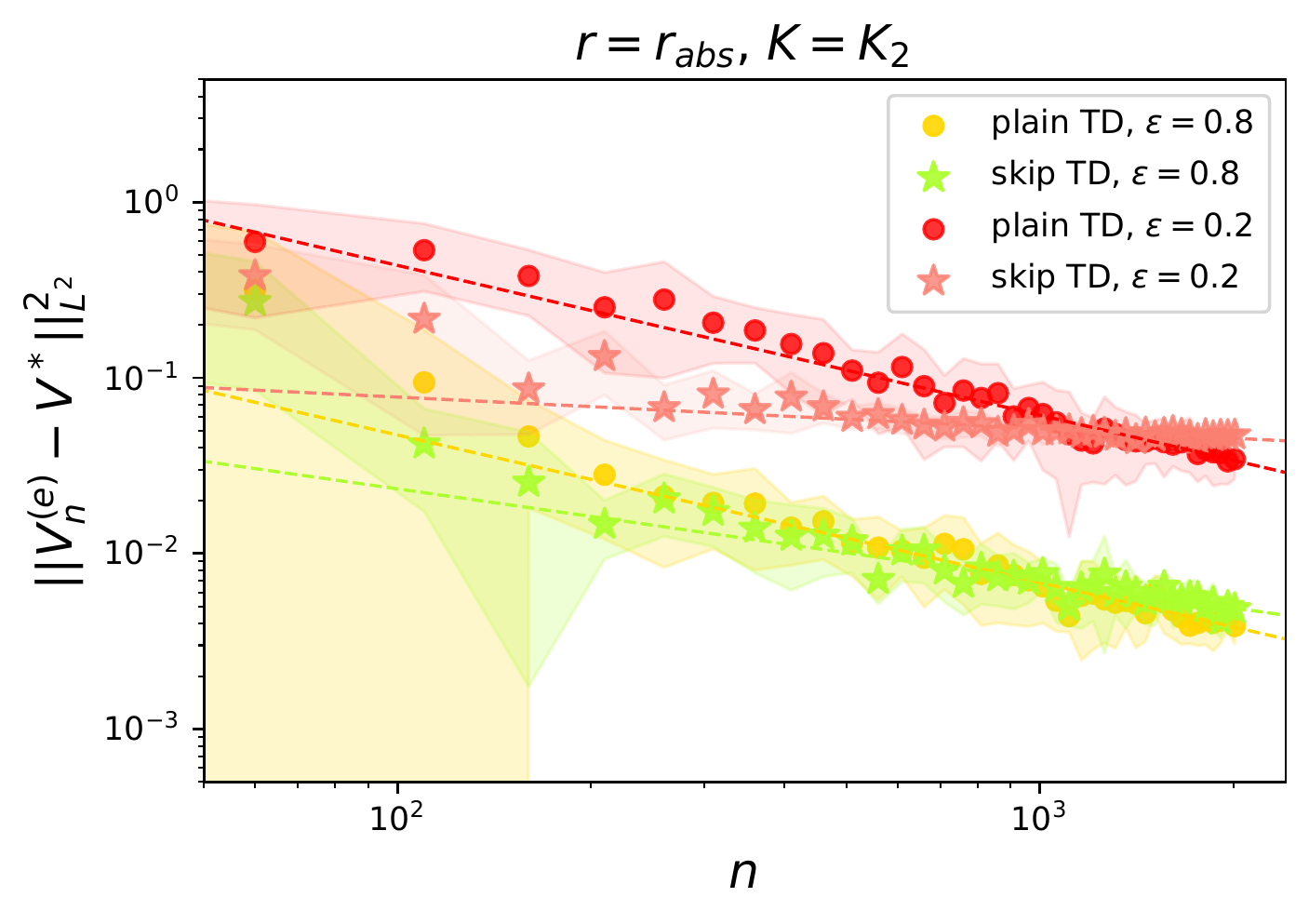}
    \caption{TD vs $\tau$-Skip-TD with fast ($\e=0.8$) and slowly ($\e=0.2$) mixing Markov chains}
    \label{fig3}
\end{figure}

\end{document}